\documentclass[10pt]{amsart}
\usepackage{mathtools}
\usepackage{kpfonts}
\usepackage{amsmath}
\usepackage{amsthm}
\usepackage{amssymb}
\usepackage{mathdots}
\usepackage{pst-node}
\usepackage{enumitem}
\usepackage{algpseudocode}
\usepackage{algorithmicx}

\newtheorem{theorem}{Theorem}[section]
\newtheorem{lemma}[theorem]{Lemma}
\newtheorem{proposition}[theorem]{Proposition}

\newtheorem{example}[theorem]{Example}

\theoremstyle{remark}
\newtheorem{remark}[theorem]{Remark}

\numberwithin{equation}{section}

\def\dim{\mathop{\rm dim}\nolimits}

\def\Ima{\mathop{\rm Im}\nolimits}

\def\Ham{\mathop{\rm Ham}\nolimits}
\def\Skew{\mathop{\rm Skew}\nolimits}
\def\Sp{\mathop{\rm Sp}\nolimits}
\def\O{\mathop{\rm O}\nolimits}
\def\C{\mathop{\rm C}\nolimits}

\def\GL{\mathop{\rm GL}\nolimits}
\def\G{\mathop{\rm G}\nolimits}

\def\R{\mathop{\rm R}\nolimits}
\def\L{\mathop{\rm L}\nolimits}

\def\V{\mathop{\rm V}\nolimits}

\parskip=\smallskipamount

\makeatletter
\renewcommand*\env@matrix[1][*\c@MaxMatrixCols c]{%
  \hskip -\arraycolsep
  \let\@ifnextchar\new@ifnextchar
  \array{#1}}
\makeatother

\begin{document}

\title[Isotropy groups]{Centralizers of the complex orthogonal and symplectic group}
\author{Tadej Star\v{c}i\v{c}}
\address{Faculty of Education, University of Ljubljana, Kardeljeva Plo\v{s}\v{c}ad 16, 1000 Lju\-blja\-na, Slovenia}
\address{Institute of Mathematics, Physics and Mechanics, Jadranska
  19, 1000 Ljubljana, Slovenia}
\email{tadej.starcic@pef.uni-lj.si}
\date{May 10, 2026}



\begin{abstract} 
We find a recursive algorithm for computing the precise centralizers of the complex orthogonal and symplectic groups, and hence the isotropy groups, with respect to the similarity transformation on the spaces of skew-symmetric and Hamiltonian matrices, respectively. These groups are conjugate to groups of certain nonsingular block matrices whose blocks are rectangular block Toeplitz.
\end{abstract}

\maketitle

\section{Introduction} \label{Sec1}

Let $\mathbb{F}$ be either the field $\mathbb{C}$ or $\mathbb{R}$.
Denote by 
$\O_n(\mathbb{F})\subset \GL_n(\mathbb{F})$ the subgroup of standard \emph{orthogonal}
matrices in the group of nonsingular $n$-by-$n$ matrices $\GL_n(\mathbb{F})$;
$R$ is orthogonal precisely when $R^{-1}=R^{T}$. 
Also, let $\Sp_{2n}(\mathbb{F})\subset \GL_{2n}(\mathbb{F})$ be the standard \emph{symplectic group} consisting of all $2n$-by-$2n$ matrices $R$ such that $R^{-1}=J_n^{-1}R^{T}J_n$, in which 
$J_n:=\begin{bsmallmatrix}
0 & I_n\\
-I_n & 0
\end{bsmallmatrix}$, where $I_n$ is the $n$-by-$n$ identity matrix.

If a matrix group $\G$ is either $\O_n(\mathbb{F})$ or $\Sp_{2n}(\mathbb{F})$ we consider its natural $\G$-\emph{con\-ju\-gation ($\G$-similarity)} action:
\begin{align}\label{ca} 
\G \times \G \to \G,\qquad (Q,R)\mapsto Q^{-1}RQ.
\end{align}
A fundamental object associated with 
this action is its \emph{centralizer (stabilizer) group} at any element $R\in\G$; see, e.g., 
\cite{Milne} for basic properties:
\begin{align}\label{CGR}
\C_{\G}(R):=\{Q\in \G \mid Q^{-1}R Q=R\}. 
\end{align}
There exists a unique multiplicative Jordan-Chevalley decomposition of $R=US=SU$ with $U\in \G$ unipotent (unitriangularizable) and $S\in \G$ semisimple (diagonalizable over $\mathbb{C}$), and we have $\C_{\G}(R)=\C_{\G}(S)\cap \C_{\G}(U)$. 
It is well known that 
the centralizer $\C_{\G}(S)$ for $\G=\O_{n}(\mathbb{C})$ and $\G=\Sp_{2n}(\mathbb{C})$, respectively,
is conjugate (via some nonsingular $P$) to a direct sum of the form
$\O_k(\mathbb{C})\oplus \O_m(\mathbb{C}) \oplus \bigoplus_{j}^{}\GL_{n_j}(\mathbb{C})$ 
and $\Sp_{2k}(\mathbb{C})\oplus \Sp_{2m}(\mathbb{C}) \oplus \bigoplus_{j}^{}\GL_{n_j}(\mathbb{C})$, 
while $\C_{\G}(U)$ is decomposed into a semidirect product (see 
monographs \cite[Theorem 3.1]{Seitz}, \cite[Ch. E.IV.2]{Spring}): 
\[
\C_{\G}(U)=\L (\C_{\G}(U)) \ltimes \R_u(\C_{\G}(U)),
\]
in which $\R_u(\C_{\G}(U))$ is a unipotent radical (i.e. a unipotent normal subgroup), and $\L (\C_{\G}(U))$ is a reductive part  
conjugate (via $P'$ possibly $\not =$ $P$) to the direct sum:
\[
\L (\C_{\G}(U))\cong
\left\{\begin{array}{ll}
\bigoplus_{k \textrm{ odd}}^{}\O_{k}(\mathbb{C})\oplus \bigoplus_{k \textrm{ even}}^{}\Sp_{k}(\mathbb{C}), & \G=\O_{n}(\mathbb{C})\\
\bigoplus_{k \textrm{ odd}}^{}\Sp_{2k}(\mathbb{C})\oplus \bigoplus_{k \textrm{ even}}^{}\O_{k}(\mathbb{C}), & \G=\Sp_{2n}(\mathbb{C})
\end{array}
\right.\hspace{-2mm}.
\]
Also, the dimension of $\C_{\G}(U)$
has already been obtained.

In this paper we give a detailed description of $\R_u(\C_{\G}(U))$ for unipotent $U\in\G$, where $\G$ is either $\O_n(\mathbb{F})$ or $\Sp_{2n}(\mathbb{F})$. 
We show that centralizers are conjugate to groups of nonsingular block matrices with rectangular block Toeplitz blocks. 
We also provide a recursive algorithm for computing $\C_{\G}(R)$ for arbitrary $R\in \G$, with $\G$ either $\O_n(\mathbb{C})$ or $\Sp_{2n}(\mathbb{C})$. 
Moreover, these centralizers are conjugate to those of so-called \emph{(indefinite) generalized} orthogonal and symplectic group (Theorems \ref{stabw1}, \ref{stabw2}). 
We add that the situation in the real case is more involved (Remark \ref{Rfw}).

Given a nonsingular matrix $H$ which is either of size $n \times n$ and symmetric ($H^T=H$) or of size $2n \times 2n$ and skew-symmetric ($H^T=-H$), let $\O_{n}(\mathbb{F},H)\subset \GL_{n}(\mathbb{F})$ and $\Sp_{2n}(\mathbb{F},H)\subset \GL_{2n}(\mathbb{F})$, respectively, be the groups of $H$-orthogonal and $H$-symplectic matrices $R$ (both satisfying $H=R^{T} H R$). 
By writing $H=P^T I_n P$ or $H=P^T J_n P$, respectively, for some 
nonsingular $P$ (see, \cite[Corollaries 4.4.4, 4.4.19]{HornJohn}),
we obtain $\O_{n}(\mathbb{C},H)=P^{-1} (\O_{n}(\mathbb{C}))P^{}$ or $\Sp_{2n}(\mathbb{C},H)=P^{-1} (\Sp_{2n}(\mathbb{C})) P^{}$, and we have
\[
C_{\O_{n}(\mathbb{C})}(R) = P^{-1}(C_{\O_{n}(\mathbb{C},H)}(P^{-1} R P^{}))P^{} \quad\textrm{or}\quad
C_{\Sp_{n}(\mathbb{C})}(R) = P^{-1}(C_{\Sp_{n}(\mathbb{C},H)}(P^{-1} R P^{}))P^{}.
\]

If $H$ is $n$-by-$n$ nonsingular symmetric, denote by $\Skew_n (\mathbb{F},H)$ the 
$\mathbb{F}$-vector space of $H$-\emph{skew-symmetric} matrices; $A$ is $H$-skew-symmetric if and only if $A^{T}H = - H A$. Similarly, for $H$ nonsingular skew-symmetric of size $2n \times 2n$ let $\Ham_{2n}(\mathbb{F},H)$ be the $\mathbb{F}$-vector space of $H$-\emph{Hamiltonian} matrices consisting of all $A$ such that $A^{T}H = - H A$ as well. In case $H=I_n$ or $H=J_n$ these notions yield the standard skew-symmetric matrices $\Skew_n(\mathbb{F})$ and Hamiltonian matrices $\Ham_n(\mathbb{F})$.

If either $\G=\O_n(\mathbb{F},H)$, $\V=\Skew_n(\mathbb{F},H)$ or $\G=\Sp_{2n}(\mathbb{F})$, $\V=\Ham_{2n}(\mathbb{F},H)$ we consider the
$\G$-similarity action on $\V$:
\begin{align}\label{cb}  
\G \times \V \to \V,\qquad (Q,A)\mapsto Q^{-1}AQ.
\end{align}
In the language of Lie groups and Lie algebras, this action is the adjoint representation of $\O_n(\mathbb{F},H)$ or $\Sp_{2n}(\mathbb{F},H)$ as a classical Lie group on its Lie algebra $\Skew_n(\mathbb{F},H)$ or $\Ham_{2n}(\mathbb{F},H)$. For more on the theory of classical groups see, for example, a monograph \cite{Weyl}. We note that
(\ref{cb}) for $\G=\O_n(\mathbb{F})$, $\V=\Skew_n(\mathbb{F})$ also represents a congruence transformation 
and these play a central role in the study of the geometry $\Skew_n(\mathbb{F})$, as shown by Hua \cite[Theorem 10]{Hua45} 
(cf. \cite
{Wan}).

Furthermore, denote the \emph{isotropy (or stabilizer) group} with respect to (\ref{cb}) by
\begin{align}\label{SGA}
\Sigma_{\G}(A):=\{Q\in \G \mid Q^{-1}A Q=A\}, \qquad A\in\V.
\end{align}
If again either $H=P^T I_n P$ or $H=P^T J_n P$ for some nonsingular $P$, it follows that $P^{-1} \Skew_n(\mathbb{F}) P=\Skew_n(\mathbb{F},H)$ 
and $P^{-1} \Ham_n(\mathbb{F}) P = \Ham_n(\mathbb{F},H)$,
we have
\[
\Sigma_{\O_{n}(\mathbb{F},H)}(P^{-1} A P)=P \Sigma_{\O_{n}(\mathbb{F})}(A) P^{-1} \quad\textrm{or}\quad
\Sigma_{\Sp_{2n}(\mathbb{F},H)}(P^{-1} A P)=P \Sigma_{\Sp_{2n}(\mathbb{F})}(A) P^{-1}.
\]

Next, the exponential map $\exp\colon A\mapsto \sum_{j=0}^{\infty}\frac{1}{j!}A^j$ 
maps $H$-skew-symmetric ($H$-Hamiltonian) matrices to $H$-orthogonal ($H$-symplectic) matrices, and nilpotent matrices 
onto unipotent matrices.
If $A$ has only one eigenvalue $\lambda$, then it is maped to $\exp (A)$ with one eigenvalue $e^{\lambda}$, and we have 
\begin{equation}\label{eqAexp}
\Sigma_{\G}(A)=C_{\G}(\pm \exp (A)).
\end{equation}
Thus, for $N$ nilpotent $H$-skew-symmetric ($H$-Hamiltonian), there is precisely one unipotent $H$-orthogonal ($H$-symplectic) $U$ with $U=\exp (N)$, and $\Sigma_{\G}(N)=C_{\G}(U)$.

Taking into account the spectral properties of matrices,
we observe that the computation of centralizers to $H$-orthogonal ($H$-symplectic) matrices is equivalent to the computation of isotropy groups of $H$-skew-symmetric ($H$-Hamiltonian) matrices. 
Our general approach to tacle this problem is related to that in our papers \cite{TSOC} and \cite{TSOS}, 
in which the iso\-tro\-py gro\-ups under ortho\-gonal *congruence on Hermitian and orthogonal similarity 
on symmetric matrices were studied, respectively. However, the\-re are several quite subtle differences between 
the problems considered in the present paper and our previous work. These differences make the ana\-ly\-sis considerably more involved and give rise to certain new phenomena.

Finally, to some extent isotropy groups 
could help us solve the problem of simultaneous 
reduction under congruence for a quadruple $(A,B,C,D)$ with $A,B$ skew-symmetric and $C,D$ nonsingular symmetric. 
By applying the Autonne-Taka\-gi factorization and using a suitable orthogonal congruence transformation we 
first obtain $(A', K, I, D')$ with the identity $I$ and the skew-symmetric normal form $K$. 
Next, $A',D'$ are simplified by using the isotropy group of $K$. This 
might have a potential application to a system of transpose-Sylvester matrix equations $AX+X^TA=C$, $BX+X^TB=D$; 
for more on these equations see \cite{TeranDopi2}, \cite{DK}, \cite{DKS}.

\section{The main results}\label{secIG}

We recall the normal forms for $H$-skew-symmetric, $H$-orthogonal, $H$-Ha\-mil\-to\-nian and $H$-symplectic matrices under $H$-orthogonal or $H$-symplectic similarity, and with respect to appropriately chosen $H$; see, e.g., \cite{Gant}, \cite{Laub}, \cite{Mehl}, \cite{Well}. Note that $H$-skew-symmetric or $H$-orthogonal ($H$-Hamiltonian or $H$-symplectic) matrices
are $H$-orthogonally ($H$-symplectically) similar if and only if they are similar (\cite[Corollary 22]{HornMerino}). For tridiagonal skew-symmetric normal forms check \cite{Djok2}.

Let us denote the basic Jordan block with the eigenvalue $\lambda\in \mathbb{C}$ and the exchange matrix
with alternating signs, respectively, by
\begin{align}\label{Jblock}
  J_m(\lambda):=\begin{bmatrix}
                                                      \lambda    &  1       & \;     & 0    \\
						      \;     & \lambda     & \ddots & \;    \\     
						      \;     & \;      & \ddots &  1     \\
                                                      0     & \;      & \;     & \lambda   
                                   \end{bmatrix},\,\,\																	
																	\lambda\in \mathbb{C},\quad
\Gamma_m:=\begin{bmatrix}
0  &                &      & 1\\
 &              &   -1     &  \\
   &  \iddots  &  &  \\
(-1)^{m}  &           &     & 0 \\
\end{bmatrix}																													
	\qquad																(m\textrm{-by-}m).
\end{align}
First, suppose $K$ is a nonsingular symmetric matrix. The Jordan canonical form of a $K$-skew-symmetric matrix $A$ 
can only contain blocks of the form $J_m(0)$ with $m$ odd, and pairs of blocks of the form $J_m(\lambda)\oplus J_m(-\lambda)$,
in which either $\lambda\in \mathbb{C}^{*}$ (nonzero complex) or $\lambda=0$ and $m$ is even.
%
%
%
For a suitable choice of $K$ (see (\ref{symmK})), the $K$-skew-symmetric normal form for $A$ under ($K$-orthogonal) similarity is of the form
\small
\begin{align}
\mathcal{S}k_{}(A)  & =
\bigoplus_{j}^{} 
\big(J_{\alpha_j}(\lambda_j)\oplus -(J_{\alpha_j}(\lambda_j))^{T}\big)
\oplus
\bigoplus_{k}^{} 
\big(J_{2\beta_k}(0)\oplus -(J_{2\beta_k}(0))^{T}\big)
\oplus
\bigoplus_l 
J_{2\gamma_l-1} (0),\,\,\, \lambda_j\in \mathbb{C}^{*},\nonumber\\
& \label{symmK}
K:=\bigoplus_{j}^{} 
\begin{bsmallmatrix}
0 & I_{\alpha_j}\\
I_{\alpha_j} & 0
\end{bsmallmatrix}
\oplus
\bigoplus_{k}^{} 
\begin{bsmallmatrix}
0 & I_{2\beta_k}\\
I_{2\beta_k} & 0
\end{bsmallmatrix}
\oplus
\bigoplus_l 
\Gamma_{2\gamma_l-1}.
\end{align}
\normalsize
Next, the Jordan canonical form of a $K$-orthogonal matrix $R$ can contain blocks $J_m(\pm 1)$ with $m$ odd, and pairs of block $J_{m}(\lambda)\oplus J_{m}(\lambda^{-1})$, in which either $\lambda\in \mathbb{C}\setminus \{0,\pm 1\}$ or $\lambda=\pm 1$ and $m$ is even.
%
%
%
Its $K$-orthogonal similarity normal form is thus of the form ($\Ima (\varphi_j)\in (0,\pi)\cup (\pi,2\pi) $):
\small
\begin{align*}
\mathcal{O}_{}(R)
= 
& 
\bigoplus_j  
\big(\exp({J_{\alpha_j}(\varphi_j)})\oplus \exp({-(J_{\alpha_j}(\varphi_j))^T}) \big)
\oplus 
\bigoplus_k 
\big( \exp({J_{2\beta_k}(0)}) \oplus \exp({-(J_{2\beta_k}(0)})^T)\big)
\\
&
\oplus\bigoplus_l \exp({J_{2\gamma_l-1}})
 \oplus \bigoplus_m 
-\big( \exp({J_{2\beta_m}(0)}) \oplus \exp({-(J_{2\beta_m}(0)})^T)\big)
\oplus\bigoplus_n  
- \exp({J_{2\epsilon_n-1}}).
\end{align*}
\normalsize

We proceed with a suitably chosen skew-symmetric matrix $J$ (see (\ref{Gm})). The Jordan canonical form of a $J$-Hamiltonian matrix $A$ 
can only contain blocks of the form $J_m(0)$ with $m$ even, and pairs of blocks of the form $J_m(\lambda)\oplus J_m(-\lambda)$,
in which either $\lambda\in \mathbb{C}^{*}$ or $\lambda=0$ and $m$ is odd.
%
%
The $J$-Hamiltonian normal form for $A$ under $J$-symplectic similarity is then of the form:
\small
\begin{align}
\mathcal{H}_{}(A) & =
\bigoplus_{j}^{} 
\big(J_{\alpha_j}(\lambda_j)\oplus -(J_{\alpha_j}(\lambda_j))^{T}\big)
\oplus
\bigoplus_{k}^{} 
\big(J_{2\beta_k-1}(0)\oplus -(J_{2\beta_k-1}(0))^{T}\big)
\oplus
\bigoplus_l 
J_{2\gamma_l},\,\,\lambda_j \in \mathbb{C}^{*},\nonumber
\\
\label{Gm}
& 
J:=\bigoplus_{j}^{} 
\begin{bsmallmatrix}
0 & I_{\alpha_j}\\
-I_{\alpha_j} & 0
\end{bsmallmatrix}
\oplus
\bigoplus_{k}^{} 
\begin{bsmallmatrix}
0 & I_{2\beta_k-1}\\
-I_{2\beta_k-1} & 0
\end{bsmallmatrix}
\oplus
\bigoplus_l
\Gamma_{2\gamma_l}.
\end{align}
\normalsize
Further, the Jordan canonical form of a $J$-symplectic matrix $R$ can 
contain blocks $J_m(\pm 1)$ with $m$ even, and pairs of blocks $J_{m}(\lambda)\oplus J_{m}(\lambda^{-1})$, 
in which either $\lambda\in \mathbb{C}\setminus \{0,\pm 1\}$ or $\lambda=\pm 1$ and $m$ is odd.
%
%
%
The $J$-symplectic similarity normal form is then as follows (and all $\Ima (\varphi_j)\in (0,\pi)\cup (\pi,2\pi) $):
\small
\begin{align*}
\mathcal{S}p_{}(R)
= 
& 
\bigoplus_j  
\big(\exp({J_{\alpha_j}(\varphi_j)})\oplus \exp({-(J_{\alpha_j}(\varphi_j))^T}) \big)
\oplus 
\bigoplus_k 
\big( \exp({J_{2\beta_k-1}(0)}) \oplus \exp({-(J_{2\beta_k-1}(0)})^T)\big)
\nonumber\\
&
\oplus\bigoplus_l \exp({H_{2\gamma_l}})
 \oplus \bigoplus_m 
\big( \exp({J_{2\beta_m-1}(0)}) \oplus \exp({-(J_{2\beta_m-1}(0)})^T)\big)
\oplus\bigoplus_n  
- \exp({H_{2\epsilon_n}}).
\end{align*}
\normalsize
%

Summands $\pm e^{A_j}$ of $H$-orthogonal or $H$-symplectic normal forms are chosen as exponentials of $H$-skew-symmetric or $H$-Hamiltonian normal blocks $A_j$, 
since it is more suitable for our applications. Note that Cayley transforms $(I-A_j)(I+A_j)^{-1}$ is a fine choice as well, provided that all $I+A_j$ are nonsingular.

By transforming $Q^{-1}MQ=M$ 
into the Sylvester equation $\mathcal{J}(M)X=X\mathcal{J}(M)$, with $X=PQP^{-1}$ for some transition matrix $P$ (cf. Proposition \ref{resAoXXA}), 
it follows directly that the centralizers and isotropy groups, 
defined by (\ref{CGR}) and (\ref{SGA}), satisfy the following properties.

\begin{proposition} \label{stabs11}
Let $\mathbb{F}$ be either $\mathbb{C}$ or $\mathbb{R}$, and let $\G \subset \GL_{n}(\mathbb{F})$ with $n=\sum_{j=1}^m n_j$.
For $j\in \{1,\ldots,m\}$ let $\G_j\subset \GL_{n_j}(\mathbb{F})$ and $H_j\in \GL_{n_1}(\mathbb{F})$, 
and set $H:=\bigoplus_{j=1}^m H_j$.
\begin{enumerate}[label= \arabic*.,ref=\arabic*,
leftmargin=20pt]
\item \label{stabs11a}
Let $\pm\lambda_1,\ldots,\pm\lambda_m\in \mathbb{C}$ be pairwise distinct eigenvalues of $A_1,\ldots,A_m$, respectively; $A:=\bigoplus_{j=1}^{m}A_j^{}$. 
Assume that either $\G=\O_n(\mathbb{F},H)$ and 
$H_j^T=H_j$, $\G_j=\O_{n_j}(\mathbb{C},H_j)$,  
$A_j\in \Skew_{n_j}(\mathbb{F},H_j)$ for all $j$,
or $G=\Sp_n(\mathbb{F},H)$ and $H_j^T=-H_j$, $\G_j=\Sp_{n_j}(\mathbb{F},H_j)$, $A_j\in \Ham_{n_j}(\mathbb{F},H_j)$ with $n_j$ even for all $j$.
Then $\Sigma_{\G}(A)=\bigoplus_{j=1}^{m}\Sigma_{\G_j}(A_j)$.
\item \label{stabs11ao}                                                                                                      
Suppose $\lambda_1,\ldots,\lambda_m\in \mathbb{C}$ are pairwise distinct eigenvalues of $R_1,\ldots,R_m$, respectively, where $R_j\in \G_j$ for any $j$; $R=\bigoplus_{j=1}^{m}R_j^{}$.
If either $H_j^T=H_j$, $\G_j=\O_{n_j}(\mathbb{F},H_j)$ for all $j$ and $\G=\O_n(\mathbb{F},H)$ or $H_j^T=-H_j$, $\G_j=\Sp_{n_j}(\mathbb{F},H_{j})$ with $n_j$ even for all $j$ and $\G=\Sp_n(\mathbb{F},H)$, then $\C_{\G}(R)=\bigoplus_{j=1}^{m}\C_{\G}(R_j)$.
\end{enumerate}

\end{proposition}

To present a detailed description of isotropy groups we use matrices of a special block structure. Given $\alpha=(\alpha_1,\ldots,\alpha_N)$ with 
$\alpha_{1}>\ldots >\alpha_{N}$ 
and $\mu=(m_1,\ldots,m_N)$, 
let $\mathbb{T}^{\alpha,\mu}$ be formed of all nonsingular
$N$-by-$N$ block matrices with \emph{rectangular upper triangular Toeplitz} blocks of size $\alpha_r$-by-$\alpha_s$:
\begin{align}\label{0T0}
\mathcal{X}=[\mathcal{X}_{rs}]_{r,s=1}^{N}\in \mathbb{T}^{\alpha,\mu},\quad
&\mathcal{X}_{rs}=
\left\{
\begin{array}{ll}
\hspace{-2mm}[0\quad \mathcal{T}_{rs}], & \hspace{-1mm} \alpha_r<\alpha_s\\
\hspace{-2mm}\begin{bmatrix}
\mathcal{T}_{rs}\\
0
\end{bmatrix}, & \hspace{-1mm} \alpha_r>\alpha_s\\
\hspace{-2mm}\mathcal{T}_{rs},& \hspace{-1mm} \alpha_r=\alpha_s
\end{array}\right. \hspace{-1mm}, \quad
\mathcal{T}_{rs}=T(A_0^{rs},\ldots,A_{b_{rs}-1}^{rs}),
\end{align}
where 
$b_{rs}:=\min\{\alpha_s,\alpha_r\}$, $A_0^{rr}\in \GL_{m_r}(\mathbb{F})$, $A_j^{rs}\in \mathbb{F}^{m_r\times m_s}$. 
By $ \mathbb{F}^{m\times n}$ we denote the set of $m$-by-$n$  complex matrices, and 
let a block \emph{upper triangular Toeplitz} matrix be
\begin{align}\label{UTT}
T(A_0,\ldots,A_{n-1}):=\begin{bmatrix}
  A_{0} & A_{1}                       & \ldots &    A_{n-1}  \\
0       & \ddots   &  \ddots    & \vdots \\
 \vdots &   \ddots           & \ddots   &   A_1 \\
0              & \ldots  &  0      & A_0
\end{bmatrix},
\end{align}
where $A_0 ,\ldots,A_{n-1}$ are of the same size, $T(A_0,\ldots,A_{n-1})=[T_{jk}]_{j,k=1}^{n}$ with $T_{jk}=0$ for $j>k$ and $T_{(j+1)(k+1)}=T_{jk}$. When $A_0$ is the identity matrix, (\ref{UTT}) is upper \emph{unitriangu\-lar} Toeplitz. For example, (\ref{0T0}) for $N=\alpha_2=2$, $\alpha_1=3$ is of the form:
\small
\[
\begin{bmatrix}[ccc|cc]
A & B & C &  F  &  G \\ 
0   & A & B     &  0  &   F \\
0   & 0   & A     &  0 & 0  \\ 
\hline
0     & H & J                         &  D  &  E \\  
0    & 0   & H              &  0    &   D    
\end{bmatrix}.
\]
\normalsize

The group structure of these matrices in (\ref{0T0}) is as follows.

\begin{proposition}\label{lemanilpo}
(\cite[Lemma 2.2]{TSOS})
The set $\mathbb{T}^{\alpha,\mu}$ defined by (\ref{0T0}) is a group.
Moreover, $\mathbb{T}^{\alpha,\mu}=\mathbb{D}\ltimes \mathbb{U}$ is a semidirect product a subgroup $\mathbb{D}\subset \mathbb{T}^{\alpha,\mu}$ containing nonsingular block diagonal matrices, and a unipotent normal subgroup $\mathbb{U}\subset \mathbb{T}^{\alpha,\mu}$ consisting of all matrices having upper unitriangular Toeplitz diagonal blocks.
\end{proposition}

Furthermore, additional conditions on $\mathcal{X}\in \mathbb{T}^{\alpha,\mu}$ in (\ref{0T0}) may be required:
\begin{enumerate}[label={\bf (\Roman*)},ref={\Roman*},wide=0pt,itemsep=4pt]
\item \label{stabs0} 
The nonzero entries of $\mathcal{X}_{rs}$ for $r>s$ are taken freely, and
$(\mathcal{X}_{rr})_{11}$ is 
$B_r$-or\-tho\-gonal ($B_r$-symplectic) for a given nonsingular symmetric (skew-symmetric) $B_r$.

\item \label{stabs4} 
For $j\geq 1$ and any $r$ we have 
$(\mathcal{X}_{rr})_{1(1+j)}:=A_{j}^{rr}=A_{0}^{rr} B_{r}^{-1}Z_{j}^{r}+D_{j}^{r}$, in which 
$A_{0}^{rr},B_{r}$ are as in (\ref{stabs0}),
$Z_{j}^{r}=
\pm (Z_j^{r})^{T}$
is chosen arbitrarily, and $D_j^{r}$ depends polynomially on $A_{j'}^{r'r'}$ with $j'\in \{0,\ldots,j-1\}$, $r'\in \{1,\ldots,r\}$ and on the entries of $\mathcal{X}_{rs}$ for $r>s$. 

\item \label{stabs2} 
The nonzero entries of $\mathcal{X}_{rs}$ for $r,s\in \{1,\ldots,N\}$ with $r<s$ are uniquely determined (polynomially) by the entries of $\mathcal{X}_{rs}$ with $r\geq s$ (described above).
\end{enumerate}

The following significant examples satisfy the above conditions.

\begin{example} \label{exW} (cf. \cite[Example 3.1]{TSOS})
Fix $c\in \{1,2\}$ and suppose $\alpha_{1}>\ldots >\alpha_{N}$,
$m_1,\ldots,m_N\in \mathbb{N}$. For $r\in \{1,\ldots,N\}$ let $B_{r}$ be nonsingular matrices with
$B_{r}^{T}=(-1)^{\alpha_r+c}B_{r}$, and let 
$Z_{j}^{r}$ 
satisfy $(Z_j^{r})^{T}=(-1)^{\alpha_r-j+c+1}Z_{r}$; all matrices are of size $m_r\times m_r$.
Define
\small
\begin{align}\label{asZ}
&\mathcal{W}=\bigoplus_{r=1}^{N}T(I_{m_r},W_1^{r},\ldots,W_{\alpha_r-1}^{r}),\qquad
W_{j}^{r}:=B_{r}^{-1}\big(Z_{j}^{r}-\frac{1}{2}\sum_{k=1}^{j-1}(-1)^k(W_k^{r})^{T}B_{r}W_{j-k}^{r}
\big).
\end{align}
\end{example}
\normalsize

\begin{example}
(cf. \cite[Example 3.2]{TSOS}) Assume $c\in \{1,2\}$ and $1\leq p< t\leq N$ with $\alpha_t<\alpha_p$, and let $B_{r}\in \GL_{m_r}(\mathbb{F})$ with $B_r^T=(-1)^{\alpha_r+c}B_{r}$ for $r\in \{p,t\}$,
while $F \in \mathbb{F}^{m_t\times m_p}$ is an arbitrary matrix. We set a matrix of the form (\ref{0T0}), having the identity as a principal subma\-trix formed of all blocks except those at the $p$-th and the $t$-th columns and rows:
\small
\begin{align}
\label{Hptk}
&
\mathcal{T}_{rs}=\left\{
\begin{array}{ll}
\hspace{-1mm}\bigoplus_{j=1}^{\alpha_r}I_{m_r}, &  r=s\\
\hspace{-1mm}0,                        &  r\neq s 
\end{array}
\right. \hspace{-1mm}, \quad \{r,s\}\not\subset\{p,t\},\qquad\,
\mathcal{T}_{tp}=
N_{\alpha_t}^{k}(F ),\quad 0\leq k \leq \alpha_t-1,\\ 
&
\mathcal{T}_{rr}=T(I_{m_r},V_1^{r},\ldots,V_{\alpha_r-1}^{r}), \quad r\in \{p,t\},\qquad\quad\quad
\mathcal{T}_{pt}=
N_{\alpha_t}^{k}\big(-B_{p}^{-1}F^{T}B_{t}\big),\nonumber
\end{align}
%
\small
\[
\begin{array}{l l}
\hspace{-3mm}\begin{array}{l}
V_{j}^{p}:=\left\{\begin{array}{ll}
\hspace{-2mm}a_{n-1}B_{p}^{-1}(F^{T}B_{t}FB_{p}^{-1})^{n}B_p^{}, & \hspace{-2mm} j=n(2k+\alpha_t-\alpha_p)\\
\hspace{-2mm}0,                      & \hspace{-2mm} \textrm{otherwise}
\end{array}
\right.\hspace{-1mm}
\\
V_{j}^{t}:=\left\{\begin{array}{ll}
\hspace{-2mm} a_{n-1}B_{t}^{-1}(B_{t}FB_{p}^{-1}F^{T})^{n}B_{t}, & \hspace{-2mm} j=n(2k+\alpha_t-\alpha_p)\\
\hspace{-2mm} 0,                      & \hspace{-2mm} \textrm{otherwise}
\end{array}
\right.\hspace{-1mm}
%
\end{array}
&                                               
\hspace{-6.0mm}
a_{n}:=
\left\{
\begin{array}{ll}
\hspace{-2mm} -\frac{1}{2^{2n+1}}C(n), & \hspace{-2.5mm} \alpha_p-\alpha_t \textrm{ even}\\
\hspace{-2mm} \frac{1}{4^{n}}(-1)^{\frac{n+1}{2}}C(\frac{n-1}{2}), & \hspace{-2.5mm} n,\alpha_p-\alpha_t \textrm{ odd}\\
\hspace{-2mm} 0, &  
\hspace{-2.5mm}\textrm{otherwise},
\end{array}
\right.
\end{array}
\]
\normalsize
Here $N_{\alpha}^{k}(X)$ denotes an $\alpha\times \alpha$ matrix with $X$ on the $k$-th super-diagonal (the main diago\-nal for $k=0$) and zeros elsewhere, and 
$C(n)={\frac{1}{n+1}} {\binom{2n}{n}}$ are Catalan numbers.
If $N=2$, $\alpha_1=4$, $\alpha_2=2$, $k=1$, $c=0$, 
$B_{1}=B_{2}=I$, then:
\small
\begin{equation}\label{exXC}
\begin{bmatrix}[cccc|cc]
I_{} & 0 & 0 & -\tfrac{1}{2}F^{T}F  &  0  &  -F^{T} \\
0   & I_{} & 0   & 0  &  0  &   0\\
0   & 0   & I_{} & 0    &  0 & 0 \\
0   & 0   &  0  &  I_{}   & 0 &0 \\
\hline
0   & 0   & 0 & F                         &  I  & 0  \\
0   & 0   & 0   & 0              &  0    &   I  \\
\end{bmatrix}.
\end{equation}
\normalsize
\end{example}
\normalsize

We state our main result, which we prove in Sec.\ref{sec2}. In comparison with \cite[Theorem 2.3]{TSOC}, \cite[Theorem 3.2]{TSOS}, several subtle differences occur, and in addition the second theorem exhibits a new phenomena in the structure of the groups.

\begin{theorem}\label{stabw2} 
Assume $\alpha=(\alpha_1,\ldots,\alpha_N)$ with $\alpha_{1}>\ldots >\alpha_{N}$,
$m_1,\ldots,m_N\in \mathbb{N}$ and fix $c\in \{1,2\}$. Set $\mu=(\mu_{1},\ldots,\mu_{N})$ with
$\mu_{r}:=
\left\{
\begin{array}{ll}
\hspace{-2mm}m_r, &  \hspace{-1mm} c+\alpha_r \textrm{ even}\\
\hspace{-2mm}2m_r, & \hspace{-1mm} c+\alpha_r \textrm{ odd}
\end{array}
\right.\hspace{-2mm} $, $n:=\sum_{r=1}^m \alpha_r \mu_{r}$, 
and let $\mathbb{T}^{\alpha,\mu}$ be as in (\ref{0T0}), while $\Gamma_{\alpha_r}$ is as in (\ref{Jblock}).
Furthermore, let $\varepsilon\in \{-1,1\}$ and define
\begin{align}\label{defMH}
&\mathcal{A}^{}=
\bigoplus_{r=1}^{N}\big(\bigoplus_{j=1}^{m_r} A_{r}\big), \,\,\,\,\, A_{r}:=
\left\{
\begin{array}{ll}
\hspace{-1mm}J_{\alpha_r}(0)\oplus -(J_{\alpha_r}(0))^T,   &  \hspace{-1mm}  
c + \alpha_r \textrm{ odd}\\
\hspace{-1mm} J_{\alpha_r}(0),      &   \hspace{-1mm} 
c + \alpha_r \textrm{ even}
\end{array}
\right.\hspace{-2mm}, \qquad
\mathcal{R}=\varepsilon e^{\mathcal{A}}, \\
 \nonumber
\G= &
\left\{
\begin{array}{ll}
\hspace{-1.5mm}
\O_{n}(\mathbb{F},H),   &  \hspace{-1.5mm} 
c =1\\
\hspace{-1.5mm}
\Sp_{n}(\mathbb{F},H),      &   \hspace{-1.5mm} 
c =2
\end{array}
\right.\hspace{-2mm}
,\quad\, H=\bigoplus_{r=1}^{N}\big( \bigoplus_{j=1}^{m_r} 
H_{r}\big),
\,\,\,\, H_r:=
\left\{
\begin{array}{ll}
\hspace{-1.5mm}
\begin{bsmallmatrix}
0 & I_{\alpha_r}\\
(-1)^{c+1} I_{\alpha_r} & 0
\end{bsmallmatrix} ,   &  \hspace{-1.5mm} 
c + \alpha_r \textrm{ odd}\\
\hspace{-1.5mm}
\Gamma_{\alpha_r},      &   \hspace{-1.5mm} 
c + \alpha_r \textrm{ even}
\end{array}
\right.\hspace{-2mm}.
\end{align}
Then the centralizer of $\mathcal{R}$ and the isotropy group of $\mathcal{A}$ (under $\G$-similarity action) coincide, and 
%
%
they admit a semidirect product decomposition with the properties below:
\begin{align}\label{CSgroup}
\Sigma_{\G}({\mathcal{A}^{}_{}})=\C_{\G}({\mathcal{R}^{}_{}})= \L(\C_{\G}({\mathcal{R}^{}_{}}))  \ltimes \R_u (\C_{\G}({\mathcal{R}^{}_{}})). 
\end{align}
\begin{enumerate}
\item
The subgroup $\L(\C_{\G}({\mathcal{R}^{}_{}}))$ is conjugate (via some matrix $\Psi_G$) to a direct sum: 
\begin{align*}
\L(\C_{\G}({\mathcal{R}^{}_{}})) \cong \bigoplus_{r=1}^{N} \big(\bigoplus_{j=1}^{\alpha_r} \G_{r}\big)=:\mathbb{O}\subset \mathbb{T}^{\alpha,\mu}, \quad\,\, 
G_{r}:= \left\{
\begin{array}{ll}
\hspace{-1mm} 
\O_{m_r}(\mathbb{F}), & \hspace{-1mm}  c+\alpha_r \textrm{ even}\\
\hspace{-1mm} 
\Sp_{2m_r}(\mathbb{F}), & \hspace{-1mm}  c+\alpha_r \textrm{ odd}
\end{array}
\right.\hspace{-2mm},
\end{align*}
%
%
\item
The unipotent radical $\R_u(\C_{\G}({\mathcal{R}^{}_{}}))$ 
is conjugated (via $\Psi_G$) to a normal subgroup $\mathbb{U}\subset \mathbb{T}^{\alpha,\mu}$, generated by matrices defined by (\ref{asZ}) and 
(\ref{Hptk}) for 
\small
$B_r := \left\{
\begin{array}{ll}
\hspace{-1mm} (-1)^{c} I_{m_r}, & c+\alpha_r \textrm{ even}\\
\hspace{-1mm} 
\begin{bsmallmatrix}
0 & I_{m_r}\\
-I_{m_r} & 0
\end{bsmallmatrix}, & c+\alpha_r \textrm{ odd}
\end{array}
\right.\hspace{-2mm}$.
\normalsize
\end{enumerate}

Moreover, there exists an algorithm for computating the centralizers (isotropy gro\-ups), appearing in (\ref{CSgroup}), and these are conjugate (via $\Psi_G$) 
to a subgroup $\mathbb{X}:=\mathbb{O} \ltimes \mathbb{U} \subset \mathbb{T}^{\alpha,\mu}$,
where every $\mathcal{X}\in \mathbb{X}$ satisfies conditions (\ref{stabs0}), (\ref{stabs4}), (\ref{stabs2}) with $B_r$ as above, and
%
\begin{align*}
\dim_{\mathbb{F}} (\mathbb{X})
=
\sum_{r=1}^{N}\big(\tfrac{1}{2} \alpha_r \mu_{r}^{2}+\sum_{s=1}^{r-1} \alpha_s \mu_{r} \mu_{s}\big) 
+(-1)^{c}\sum_{\alpha_r \textrm{ odd}}\tfrac{1}{2} \mu_{r}.
\end{align*}
\normalsize
(If $\mathbb{F}=\mathbb{R}$, then $H$ is real congruent to $\pm I_n$ precisely when $N=1$, $\alpha_1=1$, i.e. $\mathcal{A}=0$, $\mathcal{R}=\pm I_n$; 
and in this case $\C_{\O_n(\mathbb{R})}(\pm I_n)=\Sigma_{\O_n(\mathbb{R})}(0)=\O_n(\mathbb{R})$.)
\end{theorem}

\begin{theorem}\label{stabw1}
For $\alpha=(\alpha_1,\ldots,\alpha_N)$ with $\alpha_{1}>\ldots >\alpha_{N}$, 
$\mu=(m_1,\ldots,m_N)$, set $n:=2\sum_{r=1}^m \alpha_r m_r$, and let $\mathbb{T}^{\alpha,\mu}$ be as in (\ref{0T0}) with $\mathbb{F}=\mathbb{C}$.
Fix $c\in \{1,2\}$ and define further
\begin{align*}
& 
\mathcal{A}^{}(\lambda):=\bigoplus_{r=1}^{N}\Big(\bigoplus_{j=1}^{m_r}\big(J_{\alpha_r}(\lambda)\oplus -(J_{\alpha_r}(\lambda))^T\big)\Big), \,\,\, \lambda\in \mathbb{C}^{*},\quad\,\,
\mathcal{R}(\widetilde{\lambda}):=e^{\mathcal{A}(\widetilde{\lambda})}, \,\,\, 
\widetilde{\lambda} \in \mathbb{C}\setminus\{i k\pi\}_{k\in \mathbb{Z}},\\
&
G= 
\left\{
\begin{array}{ll}
\hspace{-1.5mm}
\O_{n}(\mathbb{C},H),   &  \hspace{-1.5mm} 
c =1\\
\hspace{-1.5mm}
\Sp_{n}(\mathbb{C},H),      &   \hspace{-1.5mm} 
c =2
\end{array}
\right.\hspace{-2mm}, \qquad 
H=\bigoplus_{r=1}^{N} \Big( \bigoplus_{j=1}^{\alpha_r} 
\begin{bsmallmatrix}
0 & I_{\alpha_r}\\
 (-1)^{c+1} I_{\alpha_r} & 0
\end{bsmallmatrix}\Big).
\end{align*}
Then the isotropy group $\Sigma_{\G}\big({\mathcal{A}^{}(\lambda})\big)$ and the centralizer $\C_{\G}({\mathcal{R}^{}(\widetilde{\lambda})})$ 
coincide for $\lambda=\widetilde{\lambda}\in \mathbb{C}\setminus\{k\pi\}_{i k\in \mathbb{Z}}$, and for $\lambda\in \mathbb{C}^{*}$, $\widetilde{\lambda} \in \mathbb{C}\setminus\{i k\pi\}_{k\in \mathbb{Z}}$ they are both conjugate to the subgroup 
\[
\mathbb{X}:=\{\mathcal{X}\oplus (\mathcal{X}^{-1})^{T}\mid \mathcal{X}\in \mathbb{T}^{\alpha,\mu}\}\subset \mathbb{T}^{\alpha,\mu}\oplus \mathbb{T}^{\alpha,\mu};\,\quad
\dim_{\mathbb{C}}(\mathbb{X})=\sum_{r=1}^{N} m_r\big(\alpha_r m_r+2\sum_{s=1}^{r-1}\alpha_s m_s\big).
\]
\end{theorem}

\begin{remark}\label{OpI}
\begin{enumerate}[label={(\roman*)},ref={\roman*},wide=0pt,itemsep=5pt]
\item The computation of $\mathbb{X}_c$ for $c\in \{1,2\}$, appearing in Theorem \ref{stabw2}, 
is carried out explicitly by Lemma \ref{EqT}. Also, a matrix $\Psi_G$ such that $\C_{\G}({\mathcal{R}^{}_{}})=\Psi_G^{-1}\mathbb{X}_c \Psi_G $ is provided for either $c=1$, $\G=\O_n(\mathbb{F},H)$ or $c=2$, $\G=\Sp_n(\mathbb{F},H)$. Similarly, we find $\widetilde{\Psi}_G$ such that $\C_{\G}\big(\mathcal{R}^{}(\lambda)\big)=\widetilde{\Psi}_G^{-1}\mathbb{X} \widetilde{\Psi}_G $ 
in Theorem \ref{stabw1}.
\item Theorem \ref{stabw1} does not seem to follow directly from Theorem \ref{stabw2} (the 
unipotent and nilpotent case).
As noted in Sec. \ref{Sec1}, the centralizers $\C_{\G}(S)$, $\C_{\G}(U)$ of components in the Jordan-Chevalley decomposition $R=SU$ admit nice descriptions via con\-ju\-ga\-tion by different matrices, making their intersection difficult to determine.
\end{enumerate}
\end{remark}

\section{Certain block matrix equation}\label{cereq}

For $c\in \{1,2\}$, 
$\alpha=(\alpha_1,\ldots,\alpha_N)$ with $\alpha_{1}>\ldots >\alpha_{N}$ and 
$\mu=(m_1,\ldots,m_N)$, 
let 
\small
\begin{align}\label{BBF}
&\mathcal{B}=\bigoplus_{r=1}^{N}T_a\big(B_0^{r},B_1^{r},\ldots,B_{\alpha_r-1}^{r}\big),\quad \mathcal{C}=\bigoplus_{r=1}^{N}T_a\big(C_0^{r},C_1^{r},\ldots,C_{\alpha_r-1}^{r}\big),\quad \mathcal{F}=\bigoplus_{r=1}^{N}E_{\alpha_r}(I_{m_r}),
\end{align}
\normalsize
in which $B_{j}^{r},C_{j}^{r}\in  \mathbb{F}^{m_r \times m_r}$ are symmetric for $\alpha_r-j+c$ even and skew-symmetric for $\alpha_r-j+c$ odd (i.e. $(B_j^{r})^{T}=(-1)^{\alpha_r-j+c}B_j^{r}$), $B_{0}^{r},C_{0}^{r}\in \GL_{m_r}(\mathbb{F})$, 
and $E_{\alpha}(I_{m}):=
\begin{bsmallmatrix}
 0                 &      & I_{m}\\
            &   \iddots     &  \\
I_{m}            &           &  0\\
\end{bsmallmatrix}
$
is the $\alpha\textrm{-by-} \alpha$ block exchange matrix with $I_m$ on the anti-diagonal; by
\setlength{\arraycolsep}{4.4pt}
\small
\begin{align*}
T_a(A_0,\ldots,A_{\alpha-1}):=\begin{bmatrix}
  A_{0} & A_{1}         &  \ldots             & \ldots &    A_{\alpha-1}  \\
0       & -A_0 & -A_{1}   &  \ldots    & -A_{\alpha-2} \\
 \vdots & \ddots            & A_0              & \ddots &   \vdots \\ 
 \vdots &  & \ddots   & \ddots            &  (-1)^{\alpha-2}A_1\\
0       & \ldots            &  \ldots &  0      & (-1)^{\alpha-1}A_0
\end{bmatrix}
\end{align*}
\normalsize
we denote an $\alpha$-by-$\alpha$ block \emph{alternating upper triangular Toeplitz} matrix 
such that $T_a(A_0,\ldots,A_{\alpha-1})=[T_{jk}]_{j,k=1}^{\alpha}$, $T_{jk}=0$ for $j>k$, $T_{(j+1)(k+1)}=-T_{jk}$.
To prove Theorem \ref{stabw2} it is essential 
to find all $\mathcal{X}\in \mathbb{T}^{\alpha,\mu}$ 
(see (\ref{0T0})) that solve the equation
\begin{equation}\label{eqFYFIY}
\mathcal{C}=\mathcal{F}\mathcal{X}^{T}\mathcal{F}\mathcal{B} \mathcal{X}.
\end{equation}

The following lemma provides the solution of (\ref{eqFYFIY}). 
It is adapted from 
\cite[Lemma 4.1]{TSOC,TSOS},
with several important modifications to suit our setting.

\begin{lemma}\label{EqT}
Given $c\in\{1,2\}$, 
$\alpha=(\alpha_1,\ldots,\alpha_N)$ with $\alpha_{1}>\ldots >\alpha_{N}$ and 
$\mu=(m_1,\ldots,m_N)$, 
let 
$\mathcal{B},\mathcal{C}$ be as in (\ref{BBF}). Then 
$\mathcal{X}=[\mathcal{X}_{rs}]_{r,s=1}^{N}\in \mathbb{T}^{\alpha, \mu}$ (defined by (\ref{0T0})), i.e. 
\begin{equation*}\label{EqTX}
\mathcal{X}_{rs}=
\left\{
\begin{array}{ll}
\hspace{-1mm}[0\quad \mathcal{T}_{rs}], & \alpha_r<\alpha_s\\
\hspace{-1mm}\begin{bmatrix}
\mathcal{T}_{rs}\\
0
\end{bmatrix}, & \alpha_r>\alpha_s\\
\hspace{-1mm}\mathcal{T}_{rs},& \alpha_r=\alpha_s
\end{array}\right., \qquad 
\begin{array}{l}
\\
\mathcal{T}_{rs}=T\big(A_0^{rs},\ldots,A_{b_{rs}-1}^{rs}\big),\quad b_{rs}:=\min\{\alpha_s,\alpha_r\},\\
\\
A_j^{rs}\in \mathbb{F}^{m_r\times m_s},\quad A_0^{rr}\in GL_{m_r}(\mathbb{F}),
\end{array}
\end{equation*}
%
is a solution of the equation (\ref{eqFYFIY}) if and only if it
satisfies the following properties:
\begin{enumerate}[label={(\alph*)},ref={\alph*},
leftmargin=13pt,itemindent=3pt]
%
\item \label{EqT2} 
Each $A_0^{rr}$ is a solution of the equation $C_0^{r}=(A_0^{rr})^{T}B_0^{r}A_0^{rr}$. If $N\geq 2$ matrices $A_j^{rs}$ for $j\in \{0,\ldots,\alpha_{r}-1\}$, $r,s\in \{1,\ldots,N\}$ with $r>s$ can be taken arbitrarily. 

\item \label{EqT3ca} 
Assuming (\ref{EqT2}) and choosing freely the $m_r$-by-$m_r$ matrices $Z_j^{r}$ satisfying $(Z_j^{r})^{T}=(-1)^{\alpha_r -j + c+1} Z_j^{r}$ for $j\in \{1,\ldots,\alpha_r-1\}$, the remaining entries are computed as follows:
\end{enumerate}

\vspace{-5mm}
\hspace{-12mm}
\begin{algorithmic}
\vspace{1mm}
\State $\Psi^{krs}_{n}:=\sum_{l=0}^{n}\sum_{i=0}^{n-l}(-1)^{l}(A_l^{kr})^{T}B_{n-l-i}^{k}A_i^{ks}$
\small
\State 
$\xi_n^{rs}:=
\sum_{l=1}^{n-1}\sum_{i=0}^{n-l}(-1)^{l}(A_l^{rr})^{T}B_{n-l-i}^{r}A_i^{rs}
+ \sum_{i=0}^{n-1}(A_0^{rr})^{T}B_{n-i}^{k}A_i^{rs}
+\left\{\begin{array}{ll}
\hspace{-2mm} 0, & s=r\\
\hspace{-2mm} (-1)^{n}(A_n^{rr})^{T}B_{0}^{k}A_0^{rs}, &  s> r
\end{array}
\right.
$
\normalsize
\For {$j=0:\alpha_1-1$}
    \If {$r\in \{1,\ldots,N\}$,  $ j\in \{1,\ldots,\alpha_r-1\}$}
    \State $A_j^{rr}= A_0^{rr}(C_0^{r})^{-1}\Big(Z_j^{r}+C_j^r-\frac{1}{2}\big(\xi_j^{rr}
    +\sum_{k=1}^{r-1}\Psi_{j-\alpha_k+\alpha_r}^{krr}+\sum_{k=r+1}^{N}\Psi_{j-\alpha_{r}+\alpha_k}^{krr}\big)\Big)$
    \EndIf
    \For {$p=1:N-1$}
        \If {$r\in \{1,\ldots,N\}$, $j\leq \alpha_{r+p}-1$, $r+p\leq N$}
        \small
        \State \hspace{-3mm} $A_j^{r(r+p)}=-A_0^{r(r+p)}(C_0^{r})^{-1}\Big(
        \xi_j^{r(r+p)}+\sum_{k=1}^{r-1}\Psi_{j-\alpha_k+\alpha_r}^{kr(r+p)}+ \sum_{k=r+1}^{r+p}\Psi_{j}^{kr(r+p)} $\\
        \qquad \qquad \qquad \qquad \qquad \qquad \qquad \qquad \qquad \qquad \qquad \qquad $+\sum_{k=r+p+1}^{N}\Psi_{j-\alpha_{r+p}+\alpha_k}^{kr(r+p)}\Big)$   
\normalsize
        \EndIf
    \EndFor
\EndFor
\end{algorithmic}
(We define $\sum_{j=l}^{n}a_j=0$ for $l>n$, and the loop p =1 : N-1 is not performed for $N=1$.)\\
The dimension of the space of solutions $\mathcal{X}\in\mathbb{T}^{\alpha,\mu}$ of the equation (\ref{eqFYFIY}) is  
\small
\begin{align}\label{cdim}
\sum_{r=1}^{N}\big(\tfrac{1}{2}m_r^{2}\alpha_r+\sum_{s=1}^{r-1}\alpha_s m_r m_s\big) -
\sum_{\alpha_r \textrm{ odd}}\tfrac{1}{2} m_r. 
\end{align}
\normalsize

Furtermore, if $\mathcal{B},\mathcal{C}$ are real, then $\mathcal{X}$ is real 
if and only if 
the following statements hold
%
\begin{enumerate}[label={(\roman*)},ref={\roman*},itemsep=1pt,leftmargin=25pt,itemindent=0pt]
\item Matrices $B_0^{r}$ and $C_0^{r}$ in (\ref{BBF}) 
have the same inertia
for all $r\in \{1,\ldots,N\}$.
\item All matrices $A_0^{rr}$, matrices $A_j^{rs}$ with $r>s$, $j\in \{0,\ldots,\alpha_{r}-1\}$,
and $Z_j^{r}$ for $j\in \{1,\ldots,\alpha_{r}-1\}$ in 
(\ref{EqT2}) and (\ref{EqT3ca}) 
are chosen to be real.
\end{enumerate}

\end{lemma}

Observe that $\mathcal{X}$ in Lemma \ref{EqT} is consistent with the conditions (\ref{stabs0}), (\ref{stabs4}), (\ref{stabs2}) in Sec. \ref{secIG}. Next, the order of the calculation of the entries of $\mathcal{X}$ is the following:
\begin{itemize}[
wide=12pt,itemsep=2pt,leftmargin=25pt]
%
\item $\mathcal{X}_{rs}$ 
for $r>s$ (the blocks below the main diagonal of $\mathcal{X}$; chosen freely),

\item $(\mathcal{X}_{rr})_{11}=A_0^{rr}$ (the diagonals of the main diagonal blocks of $\mathcal{X}$),

\item $(\mathcal{X}_{r(r+1)})_{11}=A_{0}^{r(r+1)}$\hspace{-1mm} (the diagonals of the 1st upper off-diagonal blocks of $\mathcal{X}$),


\item[\ldots]
\item $(\mathcal{X}_{r(r+p)})_{1(j+1)}=A_{j+1}^{r(r+p)}$ with $1\leq p\leq N-r$, $1\leq j\leq \alpha_{r+p}-1$ (the $j$-th upper off-diagonals of the $p$-th upper off-diagonal blocks of $\mathcal{X}$),
\item[\ldots]
\item $(\mathcal{X}_{11})_{1(\alpha_1-1)}=A_{\alpha_1-1}^{11}$ (the last entry in the first row of $\mathcal{X}_{11}$).
\end{itemize}

\begin{proof}[Proof of Lemma \ref{EqT}]
Since $(-1)^{\alpha_r-j}(B_j^{r})^{T}=(-1)^{c}B_j^{r}$, 
one easily verifies that 
\small
\begin{align}\label{ETTE}
E_{\alpha_r}(I_{m_r})\big(T_a(B_0^{r},B_1^{r},\ldots,B_{\alpha_r-1}^{r})\big)^{T}E_{\alpha_r}(I_{m_r})
 & =
(-1)^{\alpha_r+1}T_a\big((B_0^{r})^{T},-(B_1^{r})^{T},\ldots,(-1)^{\alpha_r-1} (B_{\alpha_r-1}^{r})^{T}\big)\nonumber \\
 & =
(-1)^{c+1} T_a\big(B_0^{r},B_1^{r},\ldots,B_{\alpha_r-1}^{r}\big).
\end{align}
\normalsize
Therefore
\[
(\mathcal{F}\mathcal{X}^{T}\mathcal{F}\mathcal{B} \mathcal{X})^{T}=\mathcal{X}^{T}\mathcal{B}^{T}\mathcal{F}\mathcal{X}\mathcal{F} =
\mathcal{X}^{T}\mathcal{F}(\mathcal{F}\mathcal{B}^{T}\mathcal{F})\mathcal{X}\mathcal{F}=
(-1)^{c+1}\mathcal{F}(\mathcal{F}\mathcal{X}^{T}\mathcal{F}\mathcal{B}\mathcal{X})\mathcal{F},
\]
so it suffices to compare the blocks in the upper triangular parts of $\mathcal{F}X^{T}\mathcal{F}\mathcal{B} X$ and $\mathcal{C}$.  
Since these blocks are rectangular upper triangular Toeplitz of the same size,
it is further enough to compare their first rows.
By simplifying the notation with $\mathcal{Y}:=\mathcal{B}\mathcal{X}$ and $\widetilde{\mathcal{X}}:=\mathcal{F}X^{T}\mathcal{F}$, 
the equation (\ref{eqFYFIY}) reduces to a system of equations:
\begin{align}
\label{f1}
(\mathcal{C}_{r(r+p)})_{1j}= & \big((\widetilde{\mathcal{X}}\mathcal{Y})_{r(r+p)}\big)_{1j} \qquad \qquad \qquad (1\leq j\leq \alpha_{r+p},\quad 0\leq p \leq N-r)  \nonumber\\
= & (\widetilde{\mathcal{X}}_{rr})_{(1)}(\mathcal{Y}_{r(r+p)})^{(j)}+\sum_{k=r+1}^{N}(\widetilde{\mathcal{X}}_{rk})_{(1)}(\mathcal{Y}_{k(r+p)})^{(j)}
 +\sum_{k=1}^{r-1}(\widetilde{\mathcal{X}}_{rk})_{(1)}(\mathcal{Y}_{k(r+p)})^{(j)}.
\end{align}

First, to calculate $A_0^{rr}$ for $r\in \{1,\ldots,N\}$ we observe (\ref{f1}) for $p=0$, $j=1$. Since
\[
(\widetilde{\mathcal{X}}_{rk})_{(1)}=\left\{
\begin{array}{ll}
\begin{bsmallmatrix}
(A_0^{kr})^{T} & * & \ldots & *
\end{bsmallmatrix}, & k\geq r\\
\begin{bsmallmatrix}
0 & * & \ldots & *
\end{bsmallmatrix}, & k<r
\end{array}
\right.,\qquad
(\mathcal{Y}_{kr})^{(1)}=
\left\{
\begin{array}{ll}
\begin{bsmallmatrix}
B_0^{k}A_0^{kr} \\
0\\
\vdots\\
0
\end{bsmallmatrix}, & k\leq r\\
0, & k>r
\end{array}
\right.,
\]
\vspace{-1mm}
we deduce $\sum_{k=1}^{N}(\widetilde{\mathcal{X}}_{rk})_{(1)}((\mathcal{Y})_{kr})^{(1)}=
(A_0^{rr})^{T}B_0^{r}A_0^{rr}$, and it implies:
\begin{equation}\label{GABA}
C_0^{r}=(A_0^{rr})^{T}B_0^{r}A_0^{rr}, \qquad r\in \{1,\ldots,N\}.
\end{equation}
To get (\ref{EqT2}), we arbitrarily set $\mathcal{X}_{rs}$ for $N\geq 2$, $r>s$ (below the main diagonal of $\mathcal{X}$).

Next, we inductively compute the remaining entries.
We fix $p\in \{0,\ldots, N-1 \}$, $j\leq \alpha_r-1$ (but not $p=j=0$) and solve (\ref{f1}) to get $A_j^{r(r+p)}$ (loops $j=0:\alpha_1-1$, $p=1:N-1$ 
in (\ref{EqT3ca})),
while assuming that we have already determined $A_{n}^{rs}$ for 
%
\begin{align}\label{induA1}
&j\geq 1, n\in \{0,\ldots,j-1\}, s\geq r 
\quad\textrm{or}\quad p\geq 1, n=j, r\leq s\leq r+p-1 \\
&\textrm{ or } \quad s\leq r, n\in \{0,\ldots,b_{rs}-1\}, N\geq 2, \qquad (1 \leq r, s\leq N).\nonumber
\end{align}
\normalsize

We have
\begin{align*}
\widetilde{\mathcal{X}}_{rk}
=
E_{\alpha_r}(I_{m_r})\mathcal{X}_{kr}^{T} E_{\alpha_k}(I_{m_k})
&=
\left\{
\begin{array}{ll}
\hspace{-1mm}\begin{bsmallmatrix}
\widetilde{\mathcal{T}}_{rk}\\
0
\hspace{-1mm}\end{bsmallmatrix}, 
& \alpha_r >\alpha_k \\
\hspace{-1mm}\begin{bsmallmatrix}
0 & \widetilde{\mathcal{T}}_{rk}
\end{bsmallmatrix}, & \alpha_r<\alpha_k\\
\widetilde{\mathcal{T}}_{rk}, & \alpha_r=\alpha_k
\end{array}
\right.\hspace{-1mm}, \quad
\widetilde{T}_{rk}=
T\big((A_0^{kr})^{T},\ldots,(A_{b_{kr}-1}^{kr})^{T}\big),
\end{align*}
%
%
%
\begin{align}\label{YSP}
&\mathcal{Y}_{ks} 
=
\left\{
\begin{array}{ll}
\hspace{-1mm}\begin{bsmallmatrix}
\mathcal{S}_{ks}\\
0
\end{bsmallmatrix}, 
& \alpha_k >\alpha_s \\
\hspace{-1mm}\begin{bsmallmatrix}
0 & \mathcal{S}_{ks}
\end{bsmallmatrix}, & \alpha_k<\alpha_s\\
\hspace{-1mm}\mathcal{S}_{ks}, & \alpha_k=\alpha_s
\end{array}
\right. \hspace{-2mm}, \quad  
\mathcal{S}_{ks} 
=T_a\big(\Phi_0^{ks},\ldots, \Phi_{b_{ks}-1}^{ks}\big),\quad
\Phi_{n}^{ks}:=\sum_{i=0}^{n} B_{n-i}^{k}A_i^{ks}.
\end{align}
To simplify the computations we set ($k,r,s\in \{1,\ldots,N\}$, $n\in \{0,\ldots, b_{rs}-1\}$):
\begin{align}
\label{Prsk}
\Psi^{krs}_{n}:=
&
\left\{
\begin{array}{ll}
\hspace{-1mm} 
\begin{bsmallmatrix}
(A_0^{kr})^{T} & (A_1^{kr})^{T} & \ldots & (A_{n}^{rr})^{T} 
\end{bsmallmatrix}
\begin{bsmallmatrix}
\Phi_{n}^{ks} \\
-\Phi_{n-1}^{ks} \\
\vdots \\
(-1)^{n}\Phi_{0}^{ks} \\
\end{bsmallmatrix}, & \hspace{-1mm} n\geq 0
\\
\hspace{-1mm} 0, & \hspace{-1mm} n<0
\end{array}
\right.\\
=
&
\left\{
\begin{array}{ll}
\hspace{-1mm}  \sum_{i=0}^{n}(-1)^{i}(A_i^{kr})^{T}\Phi_{n-i}^{ks}, & \hspace{-1mm}n\geq 0\\
\hspace{-1mm} 0, & \hspace{-1mm} n<0
\end{array}
\right.\hspace{-2mm} \nonumber
\end{align}
\normalsize
and by using the fact $(B_n^{r})^{T}=(-1)^{\alpha_r-n+c}B_n^{r}$, we observe the following:
%
\begin{align}\label{simetrija}
(\Psi^{krs}_{n})^{T}
&
=\sum_{i=0}^{n}(-1)^{n-i}(A_l^{ks})^{T}\big(\sum_{l=0}^{i}( B_{i-l}^{k})^{T}A_{n-i}^{kr}\big)\nonumber
=\sum_{l=0}^{n}\sum_{i=l}^{n}(-1)^{n+\alpha_k+l+c}(A_l^{ks})^{T} B_{i-l}^{k}  A_{n-i}^{kr}
\\
&
=(-1)^{\alpha_k+c}\sum_{l=0}^{n}(-1)^{n-l}(A_{l}^{ks})^{T} \sum_{i'=0}^{n-l} B_{i'}^{k}A_{n-l-i'}^{kr}
=(-1)^{\alpha_k-n+c}\Psi^{ksr}_{n}.
\end{align}
%
Furthermore,
\begin{align}\label{psijr}
(\widetilde{\mathcal{X}}_{rk})_{(1)}(\mathcal{Y}_{k(r+p)})^{(n+1)}
&=
\left\{
\begin{array}{ll}
\Psi_{n-\alpha_{r+p}+\alpha_k}^{kr(r+p)}, &  k\geq r+p+1\\
\Psi_{n}^{kr(r+p)},                       &  r+p \geq k\geq r+1, p\geq 1\\
\Psi_{n-\alpha_k+\alpha_r}^{kr(r+p)},     &  k\leq r\\
\end{array}
\right..
\end{align}
In particular, for $\xi_j^{r(r+p)}$ defined as in the algorithm in (\ref{EqT3ca}), we deduce:
%
\begin{align}
\label{xijrp}
(\widetilde{\mathcal{X}}_{rr})_{(1)}(\mathcal{Y}_{r(r+p)})^{(j+1)}
=\xi_j^{r(r+p)}+
\left\{\begin{array}{ll}
\hspace{-2mm}(A_0^{rr})^{T}B_0^{r}A_j^{rr}+ (-1)^{j}(A_j^{rr})^{T}B_0^{r}A_0^{rr}, &  \hspace{-1mm}
p=0\\
\hspace{-2mm}(A_0^{rr})^{T}B_0^{r}A_j^{r(r+p)}, & \hspace{-1mm} p\geq 1
\end{array}
\right.\hspace{-1mm}. 
\end{align}
\normalsize

Next, we have
\small
\begin{align*}
\begin{array}{l}
(\widetilde{\mathcal{X}}_{rk})_{(1)}=
\left\{
\begin{array}{ll}
\hspace{-1mm}\begin{bsmallmatrix}
(A_0^{kr})^{T} & (A_1^{kr})^{T} & \ldots & (A_{\alpha_{r}-1}^{kr})^{T}
\end{bsmallmatrix}, & \hspace{-1mm} k\geq r\\
\\
\hspace{-1mm}\begin{bsmallmatrix}
0& \ldots & 0 &(A_0^{kr})^{T}  & \ldots &(A_{b_{kr}}^{kr})^{T} 
\end{bsmallmatrix}, & \hspace{-1mm} k<r
\end{array}
\right.\hspace{-1mm},
\\
\\
(\mathcal{Y}_{(r+p)(r+p)})^{(j+1)}=
\begin{bsmallmatrix}
\Phi_j^{r(r+p)} \\
-\Phi_{j-1}^{r(r+p)} \\
\vdots\\
(-1)^{j}\Phi_0^{r(r+p)}
\end{bsmallmatrix},
\end{array}
%
\hspace{-4mm}
(\mathcal{Y}_{k(r+p)})^{(j+1)}=
\left\{\begin{array}{ll}
\hspace{-1mm}\begin{bsmallmatrix}
\Phi_j^{r(r+p)} \\
\vdots\\
(-1)^{j}\Phi_0^{r(r+p)}\\
0\\
\vdots\\
0
\end{bsmallmatrix}, & \hspace{-4mm} r+p > k\\
\hspace{-1mm} \begin{bsmallmatrix}
\Phi_{j-\alpha_{r+p}-\alpha_k}^{r(r+p)} \\
\vdots\\
(-1)^{j-\alpha_{r+p}-\alpha_k}\Phi_0^{r(r+p)}\\
0\\
\vdots\\
0
\end{bsmallmatrix}, & \hspace{-2mm} k> r+p
\end{array}
\right.
\hspace{-2mm}.
\end{align*}
\normalsize
Hence, the last two terms in (\ref{f1}) for $j+1$ instead of $j$ ($N\geq r+1\geq 2$) are:
%
\begin{align}\label{thjrp}
\Xi (j,r,p):=
& \sum_{k=r+1}^{N}(\widetilde{\mathcal{X}}_{rk})_{(1)}(\mathcal{Y}_{k(r+p)})^{(j+1)}\\
=
&\left\{\begin{array}{ll}
\sum_{k=r+1}^{N}\Psi_{j-\alpha_{r}+\alpha_k}^{krr}, & j\geq 1, p=0\\
\sum_{k=r+1}^{r+p}\Psi_{j}^{kr(r+p)}
+\sum_{k=r+p+1}^{N}\Psi_{j-\alpha_{r+p}+\alpha_k}^{kr(r+p)}, & j\geq 0, p\geq 1. 
\end{array}
\right.,\nonumber\\
%
\label{lajrp}
\Lambda(j,r,p):=
&  \sum_{k=1}^{r-1}(\widetilde{\mathcal{X}}_{rk})_{(1)}(\mathcal{Y}_{k(r+p)})^{(j+1)}=\sum_{k=1}^{r-1}\Psi_{j-\alpha_k+\alpha_r}^{kr(r+p)}.
\end{align}
For $j,p\geq 0$ with $j+p\geq 1$ we define
\begin{equation}\label{Djrp}
D_j^{r(r+p)}:=\xi_j^{r(r+p)}+\Xi(j,r,p)+\Lambda(j,r,p).
\end{equation}
We combine (\ref{f1}) with (\ref{xijrp}), (\ref{thjrp}), (\ref{lajrp}), (\ref{Djrp}) to obtain:
\begin{align}
\label{eqATB1}
&(A_0^{rr})^{T}B_0^{r}A_j^{r(r+p)}=-D_j^{r(r+p)},\qquad p\geq 1,\\
\label{eqATB2}
&(A_0^{rr})^{T}B_0^{r}A_j^{rr}+(-1)^{j}(A_j^{rr})^{T}B_0^{r}A_0^{rr}=C_j^{r}-D_j^{rr}, \qquad j\geq 1\,\,\,\,\, (p=0).
\end{align}
From (\ref{Prsk}) we get 
$\Psi^{krr}_{j-\alpha_r+\alpha_k}=(-1)^{\alpha_r-j+c}(\Psi^{krr}_{j-\alpha_r+\alpha_k})^{T}$, thus $\xi_j^{rr}$, $\Xi(j,r,0)$, $\Lambda(j,r,0)$, $C_j^{r}-D_j^{rr}$ are symmetric (skew-symmetric) for $\alpha_r-j+c$ even (odd).

Since (\ref{GABA}) implies 
$A_0^{rr}(C_0^{r})^{-1}=
((A_0^{rr})^{T}B_0^{r})^{-1}$, 
the equation
(\ref{eqATB1}) yields $A_j^{r(r+p)}=-A_0^{rr}(C_0^{r})^{-1}D_j^{r(r+p)}$ for $p\geq 1$.
Next, we get $A_j^{rr}$ by solving (\ref{eqATB2}). It is of the form $A^{T}X+(-1)^{\alpha_r-j+c}X^{T}A=B$ with $B^{T}=(-1)^{\alpha_r -j + c}B$,
and its solution is $X=(A^{T})^{-1}(Z+\frac{1}{2}B)$ with 
$Z^{T}=(-1)^{\alpha_r -j + c+1}Z$.
We have $A^T=(A_0^{rr})^{T}B_0^{r}$ (hence
$(A^{T})^{-1}=
A_0^{rr}(C_0^{r})^{-1}$)
and $B=C_j^{r}-D_j^{rr}$, which depend only on $A_n^{rs}$ with $n,r,s$ satisfying (\ref{induA1}).

To complete the proof of the lemma it is left to sum up the dimensions (\ref{cdim}):
\small
\begin{align*}
%
&\sum_{r=1}^{N}\sum_{s=1}^{r-1}\alpha_s m_r m_s
+ \hspace{-2.7mm}\sum_{\alpha_r \textrm{ even}}\hspace{-2.5mm}
\tfrac{\alpha_r}{4}\big(
(m_r^2-m_r)+(m_r^2+m_r)\big)+\hspace{-2.7mm}
\sum_{\alpha_r \textrm{ odd}}\hspace{-2.0mm}\big(
\tfrac{\alpha_r+(-1)^{c}}{4}(m_r^2-m_r)+
\tfrac{\alpha_r-(-1)^{c}}{4}(m_r^2+m_r)\big).
\end{align*}
\normalsize
%
By simplifying it we get 
$\sum_{r=1}^{N}\big(\tfrac{1}{2}m_r^{2}\alpha_r+\sum_{s=1}^{r-1}\alpha_s m_r m_s\big) 
+(-1)^{c}\sum_{\alpha_r \textrm{ odd}}\tfrac{1}{2} m_r$.
%
\end{proof}

\begin{example}
Le us solve $\mathcal{F}\mathcal{X}^{T}\mathcal{F}\mathcal{B}\mathcal{X}=\mathcal{B}$ for $\mathcal{F}=E_3(I)\oplus E_2(I)$, $\mathcal{B}=\mathcal{C}=(J\oplus -J\oplus J)\oplus (I \oplus -I)$ with the identity $I$ and nonsingular skew-symmetric $J$, and $\mathcal{X}$ partitioned conformally to $\mathcal{B}$. 
We have:
\small
\begin{align*}
&\qquad
\mathcal{B}=(\mathcal{F}\mathcal{X}^{T}\mathcal{F})(\mathcal{B}\mathcal{X})
=\begin{bmatrix}[ccc|c c]
A_1^{T} & B_1^{T} & C_1^{T}  &  P^{T}  &  Q^{T} \\  
0   & A_1^{T} & B_1^{T}     &  0      &   P^{T} \\ 
0   & 0   & A_1^{T}      &  0   & 0   \\  
\hline
0    & M^{T}  & N^{T}    &  A_2^{T}  &  B_2^{T} \\  
 0   & 0     & M^{T}     &  0       &   A_2 ^{T}  
\end{bmatrix}
\begin{bmatrix}[ccc|cc]
JA_1 & JB_1 & JC_1  &  JM  &  JN \\  
0   & -JA_1 & -JB_1    &  0  &   -J M  \\  
0   & 0   & J A_1     &  0  & 0   \\  
\hline
0     & P  & Q                &  A_2    &  B_2 \\  
0     & 0   & -P              &  0    &  -A_2    
\end{bmatrix}=\\
&=
\begin{bmatrix}[ccc|cc]
A_1^{T}JA_1 & A_1^{T}JB_1\hspace{-0.5mm}-\hspace{-0.5mm}B_1^{T} J A_1 & A_1^{T}J C_1\hspace{-0.5mm}+\hspace{-0.5mm}C_1^{T}J A_1\hspace{-0.5mm}-\hspace{-0.5mm}B_1^{T}J  B_1    & \hspace{-0.5mm} A_1^{T}JM\hspace{-0.5mm}+\hspace{-0.5mm}P^{T}A_2 \hspace{-0.5mm} &  A_1^{T}J N+B_1^{T}JM  \\ 
& & + P^{T}Q - Q^{T} P & & + P^{T}B_2- Q^{T}A_2 \\ 
  0         & -A_1^{T}JA_1 &  -A_1^{T} J B_1 +B_1^{T}J A_1    &  0   &   -A_1^{T}JM - P^{T}A_2 \\ 
0   &     0                  & A_1^{T}J A_1                        &  0  & 0             \\  
\hline 
   &    &                     &  A_{2}^{T} A_2       &   A_2^{T}B_2 - B_2^{T} A_2 \\ 
	 &    &                     &                       &     -M^{T}JM                          \\
   &    &                       &      0           &    + A_2^{T}A_2         
\end{bmatrix}
\end{align*}
\normalsize
By comparing diagonals of the main diagonal blocks we get $J=(A_1)^{T} J A_1$, $I=(A_2)^{T}A_2$ ($A_1$ $J$-symplectic, $A_2$ orthogonal). Next, choose $P,Q$ arbitra\-rily. The diagonal element of the upper right block gives $A_1^{T}JM+P^{T}A_2=0$, 
thus $M=A_1^{} J^{-1}P^{T}A_2$.

The first superdiagonals yield $A_1^{T}JB_1-B_1^{T} J A_1=0$, $A_2^{T}B_2 - B_2^{T} A_2-M^{T}JM=0$ and $A_1^{T}JN-B_1^{T}M+P^{T}B_2-Q^{T}A_2$, thus $B_1=A_1J^{-1}Z_1$ with $Z_1$ skew-symmetric, $B_2=\frac{1}{2}A_2J^{-1}M^{T}JM+ A_2J^{-1}Z_2$ with $Z_2$ symmetric, $N=A_1 J^{-1}(B_1^{T}M-P^{T}B_2+Q^{T}A_2)$, respectively.
Finally, the second superdiagonal yields $A_1^{T}J C_1+C_1^{T}J A_1-B_1^{T}J B_1 + P^{T}Q - Q^{T} P=0$, therefore $C_1=\frac{1}{2}A_1 J^{-1}(B_1^{T}J B_1 - P^{T}Q + Q^{T} P)+A_1 J^{-1} Z_2
$ with symmetric $Z_3$.
\end{example}

Solutions of the equation (\ref{eqFYFIY}) with $\mathcal{C}=\mathcal{B}$ block diagonal have nice properties.

\begin{lemma}\label{lemauni}
Let $\mathbb{T}_{}^{\alpha,\mu}$ for $\alpha=(\alpha_1,\ldots,\alpha_n)$ with $\alpha_{1}>\ldots >\alpha_{N}$ and $\mu=(m_1,\ldots,m_N)$ be as 
in (\ref{0T0}), and let $\mathbb{X}_{}\subset \mathbb{T}_{}^{\alpha,\mu}$ be the set of solutions of the equation (\ref{eqFYFIY}) for 
$\mathcal{C}=\mathcal{B}=\mathcal{B}=\bigoplus_{r=1}^{N}\big(\bigoplus_{j=1}^{\alpha_r} (-1)^{j} B_{r}\big)$ with $B_r$ nonsingular symmetric or skew-symmetric.
Then $\mathbb{X}$ is a group, more precisely, a semidirect product
\[
\mathbb{X}_{}=\mathbb{O}_{}\ltimes \mathbb{V}_{}\subset \mathbb{T}_{}^{\alpha,\mu},
\]
in which $\mathbb{O}_{}$ consists of all $\mathcal{Q}=\bigoplus_{r=1}^{N}\big(\bigoplus_{j=1}^{\alpha_r} Q_{r}\big)$ for $Q_{r}\in \mathbb{C}^{m_r\times m_r}$ such that $
Q_{r}^{T}B_{r}Q_{r}=B_{r}:=[\mathcal{B}_{rr}]_{11}$, while any $\mathcal{V}\in \mathbb{V}_{}$ is unipotent and of the form
$\mathcal{V}=\prod_{j=0}^{n_{\mathcal{V}}}\mathcal{U}_j$, where  
$\mathcal{U}_0$ is of the form (\ref{asZ}), and $\mathcal{U}_1,\ldots,\mathcal{U}_{n_{\mathcal{V}}}$
are of the form (\ref{0T0}) with (\ref{Hptk}).
\end{lemma}

To prove the lemma we use ideas from the \cite[Lemma 4.2]{TSOS}, 
with the distinction that, in the present case, the diagonal blocks of $\mathcal{B}$ are alternating block diagonal.

\begin{proof}[Sketch of the proof of Lemma \ref{lemauni}]
It is immediate that matrices of the form (\ref{asZ})
are among solutions of the equation.
We now prove that matrices $\mathcal{T}=[\mathcal{T}_{rs}]_{r,s=1}^N$ of the form (\ref{0T0}) with (\ref{Hptk})
solve the equation
precisely when $\mathcal{T}_{rr}=T(I,V_1^r,\ldots,V_{\alpha_j}^r)$ with $r\in \{t,p\}$ and the following equations for $n\in \{1, \ldots,\alpha_r-1\}$, $r\in \{t,p\}$ are satisfied:
\small
\begin{align*}
& B_r V_j^r+\sum_{j=1}^{n-1}(-1)^{j-1}(V_{j-1}^r)^TB_r V_1^r
+(-1)^{j}(V_j^r)^TB_r=
\left\{
\begin{array}{ll}
-F^T B_t F, & j=2k+\alpha_p-\alpha_t, \,\, r=p\\
-B_t F B_p^{-1} F B_t, & j=2k+\alpha_p-\alpha_t,\,\, r=t \\
0,             & \textrm{otherwise}
\end{array}
\right..
\end{align*}
\normalsize
These equations are 
of the form $A^TV_n+(-1)^{n}V_n^T A=B$ with given $A,B$, and we take a particular solution $V_n=\frac{1}{2}(A^{T})^{-1}B$. We subsequently solve the above equations
to get $V_1$, $V_2$, \ldots $V_{\alpha_r-1}$.
\normalsize
In view of \cite[Lemma 4.2]{TSOS} one can easily observe that
\begin{align}
&
\qquad V_{j}^{r}:=\left\{\begin{array}{ll}
\hspace{-1mm}a_{n-1}B_{p}^{-1}(F^{T}B_{t}FB_{p}^{-1})^{n}B_p^{}, & j=n(2k+\alpha_p-\alpha_t),\,\, r=p\\
\hspace{-1mm} a_{n-1}B_{t}^{-1}(B_{t}FB_{p}^{-1}F^{T})^{n}B_{t}, & j=n(2k+\alpha_t-\alpha_t),\,\, r=t\\
\hspace{-1mm} 0,                      & \textrm{otherwise}
\end{array}
\right.\hspace{-1mm},\nonumber
\end{align}
where $a_n$ has the initial term $a_0:= -\frac{1}{2}$ and satisfies the functional equation: 
\begin{align*}
a_n  =  -\frac{1}{2} \sum_{j=1}^{n}(-1)^{j(2k+\alpha_p-\alpha_t)}a_{j-1} a_{n-j}
= \frac{1}{2} (-1)^{\alpha_p-\alpha_t+1}\sum_{j=0}^{n-1}(-1)^{j (\alpha_p-\alpha_t)}a_{j} a_{n-1-j}, \quad n\geq 1.
\end{align*}
Let the generating function associated with $a_n$ be $f(t):=\sum_{j=0}^{\infty}a_jt^{j}$. For $\alpha_p-\alpha_t$ even we obtain the equation $f(t)=-\frac{1}{2}t(f(t))^2-\frac{1}{2}$, which yields $f(t)=-\frac{1}{t}\big(1-(1-t)^{\frac{1}{2}}\big)$ and $a_{n}=-\frac{1}{2^{2n+1}}\frac{1}{n+1}{\binom{2n}{n}}$. For $\alpha_p-\alpha_t$ odd we get $f(t)=-\frac{1}{2}tf(t)f(-1)-\frac{1}{2}$, hence $f(t)+f(-t)=-1$ and $f(t)=\frac{1}{2}tf(t)(1+f(1))-\frac{1}{2}$, and eventually $f(t)=-\frac{1}{2t}\big(2-t-(t^2+4)^{\frac{1}{2}}\big)$ with $a_{2n}=0$ and $a_{2n+1}=\frac{(-1)^{n+1}}{2^{4n+2}(n+1)}{\binom{2n}{n}}$ for $n\geq 1$.

The rest of the proof proceeds mutatis mutandis as in \cite[Lemma 4.2]{TSOS}.
\end{proof}
\section{Proofs of the main results}\label{sec2}

Before proving Theorems \ref{stabw1} and \ref{stabw2}, we recall the classical result \cite[Ch. VIII]{Gant} on the 
solutions of a homogeneous Sylvester equation.

\begin{proposition} \label{resAoXXA} 
Suppose $\mathcal{J}$ is the Jordan canonical form (with direct summands of the form (\ref{Jblock})) and let us consider the following matrix equation:
\begin{align}\label{eqJYYJ}
\mathcal{J}X=X\mathcal{J}
\end{align}  
\begin{enumerate}[label={\arabic*.},ref={\arabic*},
leftmargin=20pt]
\item \label{resAoXXA1}
Let further $\lambda_1,\ldots,\lambda_n\in \mathbb{C}$ be pairwise distinct eigenvalues of $\mathcal{J}=\bigoplus_{j=1}^{n}\mathcal{J}_j$,
in which $\mathcal{J}_j$ consists of all summands that correspond to the eigenvalue $\lambda_j$.
Then the solution of (\ref{eqJYYJ}) is of the form $Y=\bigoplus_{j=1}^{n} X_j$, with $X_j$ a solution of $\mathcal{J}_j X_j=X_j\mathcal{J}_j$.

\item \label{resAoXXA2}
Assume $\mathcal{J}=\bigoplus_{r=1}^{N}\big(\bigoplus_{j=1}^{m_r} J_{\alpha_r}(\lambda)\big)$, $\lambda \in \mathbb{C}$ (i.e. $\mathcal{J}$ has precisely one eigenvalue $\lambda$) and $\alpha_1>\ldots>\alpha_N$. Then the solution of (\ref{eqJYYJ}) is $X=[X_{rs}]_{r,s=1}^{N}$, where every block $X_{rs}$ is further an $m_r$-by-$m$ block matrix with $\alpha_r$-by-$ \alpha_s$ blocks of the form
\begin{equation}\label{QTY}
\left\{\begin{array}{ll}
\begin{bmatrix} 
0 & T
\end{bmatrix}, & \alpha_r<\alpha_s \\
\begin{bmatrix}
T\\
0
\end{bmatrix}, & \alpha_r>\alpha_s\\
T, & \alpha_r=\alpha_s
\end{array}
\right.,
\end{equation}
with $T$ upper triangular Toeplitz matrix of size $b_{rs}\times b_{rs}$; $b_{rs}:=\min\{\alpha_r,\alpha_s\}$. 
\end{enumerate}
\end{proposition}

Proposition \ref{resAoXXA} (\ref{resAoXXA2}) remains valid also if we replace $J_{\alpha_r}(\lambda)$ by $\exp({J_{\alpha_r}(\lambda)})$.

\begin{lemma}\label{resAoXXAo} 
Assume $\mathcal{J}=\bigoplus_{r=1}^{N}\big(\bigoplus_{j=1}^{m_r} \exp({J_{\alpha_r}(\lambda)})\big)$, $\lambda \in \mathbb{C}$.
Then the solutions $X$ of the equation $\mathcal{J}X=X\mathcal{J}$ coincide with those described in Proposition \ref{resAoXXA} (\ref{resAoXXA2}). 
\end{lemma}

\begin{proof}
Since $\exp({J_{\alpha}(\lambda)})=e^{\lambda }\exp({J_{\alpha}(0)})
$ and $\exp({J_{\alpha}(0)})=S_{\alpha}J_{\alpha}(1)S_{\alpha}^{-1}$ for some transition matrix $S_{\alpha}$, the equation $\exp({J_{\alpha_r}(\lambda)})Y =Y \exp({J_{\alpha_s}(\lambda)})$ transforms to
\begin{align}
\label{eJ0XXeJ0}
&\exp({J_{\alpha_r}(0)})Y =Y \exp({J_{\alpha_s}(0)}),\\
\label{J1XXJ1}
&J_{\alpha_r}(1)\widetilde{Y} =\widetilde{Y} J_{\alpha_s}(1), \qquad \quad \widetilde{Y}=S^{-1}_{\alpha_r}YS_{\alpha_s}.
\end{align}
By Proposition \ref{resAoXXA} (\ref{resAoXXA2}) every solution $\widetilde{Y}$ of (\ref{J1XXJ1}) is of the form (\ref{QTY}), 
so the dimension of the solution space of (\ref{J1XXJ1}) (as well as (\ref{eJ0XXeJ0})) is $b_{rs}=\min\{\alpha_r,\alpha_s\}$. On the other hand, since upper triangular Toeplitz matrices commute, it is easy to verify that any matrix of the form (\ref{QTY}) solves (\ref{eJ0XXeJ0}).
Since these solutions form a subspace of (maximal) dimension $b_{rs}$, it is thus equal to the whole solution space. 
\end{proof}

We also recall a technical lemma \cite[Lemma 3.3 (1)]{TSOC} (cf. \cite[Sec. 3.1]{Lin} and \cite[Sec. 2]{TSOS}), which enables us to transform solutions of a Sylvester equation (block matrices with upper triangular Toeplitz blocks) into block upper triangular Toeplitz matrices of the form (\ref{0T0}). It relies on the permutation matrices
\begin{equation}\label{perS}
\Omega_{\alpha,m}:=\left[e_1\;e_{\alpha+1}\;\ldots\;e_{(m-1)\alpha+1}\;e_2\;e_{\alpha+2}\;\ldots\;e_{(m-1)\alpha+2}\;\ldots\;e_{\alpha}\;e_{2\alpha}\;\ldots\;e_{\alpha m}\right],
\end{equation}
where $e_1,
\ldots,e_{\alpha m}$ is the standard orthonormal basis in $\mathbb{C}^{\alpha m}$. Post-multiplication by $\Omega_{\alpha,m}$ (pre-multiplication by $\Omega_{\alpha,m}^{T}$) puts the $k$-th, the $(\alpha+k)$-th, \ldots, the $((m-1)\alpha+k)$-th column (row) together for all $ k \in \{1,\ldots, \alpha\}$. For example:
\[
\Omega_{3,2}^{T}\begin{bmatrix}[cc|cc|cc]
a_1 & b_1 & a_2 & b_2 & a_3 & b_3 \\
0   & a_1 & 0   & a_2 & 0   & a_3\\
0   & 0   & 0   & 0   & 0   &  0 \\
\hline
a_4 & b_4 & a_5 & b_5 & a_6 & b_6 \\
0   & a_4 & 0   & a_5 & 0   & a_6\\
0   & 0   & 0   & 0   & 0   &  0 
\end{bmatrix}\Omega_{2,3}
=
\begin{bmatrix}[ccc|ccc]
a_1 & a_2 & a_3 & b_1 & b_2 & b_3 \\
a_4 & a_5 & a_6   & b_4 & b_5   & b_6\\
\hline
0   & 0   & 0   & a_1   & a_2   &  a_3 \\
0 &  0 &   0 & a_4 & a_5 & a_6 \\
\hline
0   & 0 & 0   & 0 & 0   & 0\\
0   & 0   & 0   & 0   & 0   &  0 
\end{bmatrix}.
\]

\begin{lemma}\label{lemaP}
Suppose
$X=[X_{rs}]_{r,s=1}^{N}$ is an $N$-by-$N$ block matrix whose block $X_{rs}=[(X_{rs})_{jk}]_{j,k=1}^{m_r,m_s}$ is an $m_r$-by-$m_s$ block matrix with blocks of size $\alpha_r\times \alpha_s$, and such that
\[
(X_{rs})_{jk}=
\left\{
\begin{array}{ll}
[0\quad T_{jk}^{rs}], & \alpha_{r}<\alpha_{s}\\
\begin{bmatrix}
T_{jk}^{rs}\\
0
\end{bmatrix}, & \alpha_{r}>\alpha_{s}\\
T_{jk}^{rs},& \alpha_{r}=\alpha_{s}
\end{array}\right., \qquad
\begin{array}{l}
T_{jk}^{rs}=T(a_{0,jk}^{rs},a_{1,jk}^{rs},\ldots,a_{b_{rs}-1,jk}^{rs})\in \mathbb{C}^{b_{rs}\times b_{rs}},\\
\\
b_{rs}:=\min \{\alpha_r,\alpha_s\}.
\end{array}
\]
Set $\Omega=\bigoplus_{r=1}^{N}\Omega_{\alpha_r,m_r}$, in which $\Omega_{\alpha_r,m_r}$ are defined by (\ref{perS}). Then 
\vspace{-1mm}
\begin{align}\label{0T02}
&\mathcal{X}:=\Omega^{T}X\Omega, \qquad \mathcal{X}=[\mathcal{X}_{rs}]_{r,s=1}^{N},\quad
\mathcal{X}_{rs}=
\left\{
\begin{array}{ll}
[0\quad \mathcal{T}_{rs}], & 
\alpha_r<\alpha_s\\
\begin{bmatrix}
\mathcal{T}_{rs}\\
0
\end{bmatrix}, & \alpha_r>\alpha_s\\
\mathcal{T}_{rs},& \alpha_r=\alpha_s
\end{array}\right.,
\end{align}
where $\mathcal{X}_{rs}$ is of size $\alpha_r \times \alpha_s$ and $\mathcal{T}_{rs}=
T(A_0^{rs},\ldots,A_{b_{rs}-1}^{rs})$ 
with $ A_{n}^{rs}:=[a_{n,jk}^{rs}]_{j,k=1}^{m_r,m_s}$.
\end{lemma}

We outline the general approach used to prove our main results. Let $H$ be symmetric (skew-symmetric), 
and let $\mathcal{A}^{}$ be 
$H$-skew-symmetric ($H$-Hamiltonian) normal form with one distinct eigenvalue $\lambda\in \mathbb{C}$.
The isotropy group at $\mathcal{A}$ under $H$-orthogonal ($H$-symplectic) similarity 
consists of all $H$-orthogonal ($H$-symplectic) matrices $Q$ satisfying the Sylvester equation: 
\begin{equation}\label{HQQH}
\mathcal{A}^{}Q= Q\mathcal{A}^{},  \qquad \qquad (Q^T H Q=H).
\end{equation}

The first equation of (\ref{HQQH}) is transformed into a simpler form
\begin{equation}\label{JQQJ}
\mathcal{J}^{}(\mathcal{A})X= X\mathcal{J}^{}(\mathcal{A}), \quad  \quad
X=U Q U^{-1} \qquad \quad \quad
(\mathcal{A}=U^{-1}\mathcal{J}(\mathcal{A})U^{}),
\end{equation}
where $\mathcal{J}(\mathcal{A})$ is the Jordan canonical form for $\mathcal{A}$ and $U$ is a transition matrix. and 
Proposition \ref{resAoXXA} (\ref{resAoXXA2}) gives all solutions $X$ of (\ref{JQQJ}). 
Next, we determine which of these solutions provide $H$-orthogonal ($H$-symplectic) solutions $Q=U^{-1} X U$ (of the second equation (\ref{HQQH})). This yields an equation
\begin{align}\label{QK}
H  =                   & \big((U^{})^{T}X^{T} (U^{-1})^{T}\big) H  \big(  U^{-1} X U \big),  \qquad \qquad (H=Q^T H Q),\\
(U^{-1})^{T} H U^{-1}=  & X^{T} \big((U^{-1})^{T} H U^{-1}\big) X, \qquad\qquad (Q=U^{-1} X U).\nonumber
\end{align}
The core of the problem therefore lies in linking the solutions (\ref{JQQJ}) to the equation (\ref{QK}).
Eventually, it reduces to a special matrix equation involving matrices of the form (\ref{0T0}), which we may solve using Lemma \ref{EqT}.  
This part of the proof appears to be significantly more involved than in \cite{TSOC} and \cite{TSOS}.

Moreover, the isotropy groups of $H$-skew-symmetric ($H$-Hamiltonian) matrices under $H$-orthogonal ($H$-symplectic) similarity immediatelly give the centralizers of $H$-orthogonal ($H$-symplectic) group (cf. (\ref{eqAexp})). Indeed, since $\mathcal{J}(\mathcal{A})$ has one distinct eigenvalue $\lambda$, then in view of Proposition \ref{eJ0XXeJ0}, solutions of (\ref{JQQJ})
coincide with solutions of 
$\exp(\mathcal{J}(\mathcal{A}))X=X \exp(\mathcal{J}(\mathcal{A}))$. Further, $\mathcal{A}=U^{-1}\mathcal{J}(\mathcal{A})U^{}$ yields $\exp (\mathcal{A})=U^{-1}\exp (\mathcal{J}(\mathcal{A}))U^{}$, thus solutions of (\ref{HQQH}) give solutions of $(\varepsilon \exp(\mathcal{A}))Q= Q (\varepsilon \exp(\mathcal{A}))$, with $\varepsilon \in \{1,-1\}$.

\begin{proof}[Proof of Theorem \ref{stabw2}]
For $c=1$ or $c=2$, 
let $H$ and the $H$-skew-symmetric or $H$-Hamiltonian normal form $\mathcal{A}$, with eigenvalue $0$, be as in (\ref{defMH}):
\begin{align}
&\label{Hr}
H:=\bigoplus_{r=1}^{N} \big( \bigoplus_{j=1}^{m_r} 
H_{r}\big),
\qquad H_r=
\left\{
\begin{array}{ll}
\hspace{-1mm}
\begin{bsmallmatrix}
0 & I_{\alpha_r}\\
(-1)^{c} I_{\alpha_r} & 0
\end{bsmallmatrix} ,   &  \hspace{-1mm} 
c + \alpha_r \textrm{ odd}\\
\hspace{-1mm}
\Gamma_{\alpha_r},      &   \hspace{-1mm} 
c + \alpha_r \textrm{ even}
\end{array}
\right.,\\
\label{Kcase2}
&\mathcal{A}^{}=
\bigoplus_{r=1}^{N}\big(\bigoplus_{j=1}^{m_r} A_{r}\big), \quad A_{r}:=
\left\{
\begin{array}{ll}
\hspace{-1mm}J_{\alpha_r}(0)\oplus -(J_{\alpha_r}(0))^T,   &  \hspace{-1mm}  
c + \alpha_r \textrm{ odd}\\
\hspace{-1mm} J_{\alpha_r}(0),      &   \hspace{-1mm} 
c + \alpha_r \textrm{ even}
\end{array}
\right.\hspace{-2mm},
\end{align}
where $\Gamma_{\alpha_r}$ is defined by (\ref{Jblock}).
Next, the Jordan canonical form $\mathcal{J}(\mathcal{A})$ of $\mathcal{A}$ and its corresponding transition matrix $U$ (i.e. $\mathcal{A}=U^{-1}\mathcal{J}(\mathcal{A})U$) are:
\begin{align}\label{Jcase2}
\mathcal{J}(\mathcal{A})
& = \bigoplus_{r=1}^{N}\big( \bigoplus_{j=1}^{m_r} \mathcal{J}(\widetilde{A}_r) \big),
\quad \mathcal{J}(\widetilde{A}_r):=
\left\{
\begin{array}{ll}
\hspace{-1mm}J_{\alpha_r}(0) \oplus J_{\alpha_r}(0),   &  \hspace{-1mm} 
c + \alpha_r \textrm{ odd}\\
\hspace{-1mm}J_{\alpha_r}(0),             &   \hspace{-1mm} 
c + \alpha_r \textrm{ even}
\end{array}
\right.\hspace{-2mm},
\\
\label{Ur}
U
& =
\bigoplus_{r=1}^{N} \big( \bigoplus_{j=1}^{m_r} U_r \big), \qquad
U_r
:=
\left\{
\begin{array}{ll}
\hspace{-1.5mm} I_{\alpha_r}\oplus \Gamma_{\alpha_r}, & \hspace{-1.5mm} c + \alpha_r \textrm{ odd}\\
\hspace{-1.5mm} I_{\alpha_r}, &  \hspace{-1.5mm} c + \alpha_r \textrm{ even}
\end{array}
\right..
\end{align}

In view of (\ref{HQQH}) and (\ref{JQQJ}), the equation $\mathcal{A}Q=Q\mathcal{A}$ with $\mathcal{A}$ as in (\ref{Kcase2}) and with $H$-orthogonal or $H$-symplectic $Q$, transforms to $\mathcal{J}(\mathcal{A})X=X\mathcal{J}(\mathcal{A})$ for $X=U Q U^{-1}$ with $U$ as in (\ref{Ur}).
Proposition \ref{resAoXXA} (\ref{resAoXXA1}) yields the solution of this equation $X=[X_{rs}]_{r,s=1}^{N}$ with an $\mu_r$-by-$\mu_s$ block matrix $X_{rs}$ whose blocks are $\alpha_r$-by-$\alpha_s$ rectangular Toeplitz
of the form (\ref{QTY}), where 
\small
$
\mu_r:=\left\{
\begin{array}{ll}
2m_r,  & c + \alpha_r \textrm{ odd} \\
m_r, &  c + \alpha \textrm{ even}
\end{array}
\right.
$.
\normalsize

Next, $H$-orthogonality or $H$-symplecticity of $Q=U^{-1} X U$ is described by 
(\ref{QK}):
\begin{align}\label{QTQcase2}
(U^T)^{-1} H U^{-1} =   &  X^T \big((U^T)^{-1} H U^{-1}\big) X.
\end{align}
We shall now write $(U^{-1})^{T} H U^{-1}$ as a product of block diagonal matrices, 
one with diagonal and one with anti-diagonal blocks; this is the trick of the proof.

We observe that $\Gamma_{\alpha_r}^{T}=\Gamma_{\alpha_r}^{-1}=(-1)^{\alpha_r +1}\Gamma_{\alpha_r}^{}$ and we deduce
\begin{align}
\label{UTHU}
& ((I_{\alpha_r} \oplus \Gamma_{\alpha_r} )^{-1})^{T} 
\begin{bmatrix}
0 & I_{\alpha_r}\\
(-1)^{c+1} I_{\alpha_r}& 0 
\end{bmatrix}
(I_{\alpha} \oplus \Gamma_{\alpha_r})^{-1} 
= 
\begin{bmatrix}
0 &  (-1)^{\alpha_r -1} \Gamma_{\alpha_r}\\
(-1)^{c+1} \Gamma_{\alpha_r} & 0 
\end{bmatrix}= \\
=
(E_{\alpha_r} \oplus  & E_{\alpha_r}) 
\begin{bmatrix}
0 & (-1)^{\alpha_r -1} E_{\alpha_r}\Gamma_{\alpha_r}\\
(-1)^{c+1} E_{\alpha_r}\Gamma_{\alpha_r} & 0 
\end{bmatrix}
= 
(E_{\alpha_r} \oplus E_{\alpha_r})
\begin{bmatrix}
0 & -F_{\alpha_r}\\
(-1)^{c+\alpha_r+1} F_{\alpha_r} & 0 
\end{bmatrix},\nonumber
\end{align}
with the exchange matrix and the diagonal matrix with alternating $\pm 1$ denoted by:
\begin{equation}
\label{EF}
E_{\alpha_r}:=\begin{bmatrix}
0  &                &      & 1\\
 &              &   1     &  \\
   &  \iddots  &  &  \\
1  &           &     & 0 \\
\end{bmatrix}, \qquad		
  F_{\alpha_r}:=\begin{bmatrix}
                 -1    &       &     & 0    \\
						           & 1     &     &     \\     
						           &       & \ddots &       \\
                 0     &       &     & (-1) ^{\alpha-1}  
                                   \end{bmatrix},
																										\qquad						(\alpha_r\textrm{-by-}\alpha_r),
\end{equation} 
and where $\Gamma_{\alpha_r}=(-1)^{\alpha_r} E_{\alpha_r} F_{\alpha_r}$.
By combining (\ref{Hr}) and (\ref{Ur}) with (\ref{UTHU}), we get
\[
(U_r^{-1})^T H_r (U_r^{-1})=\left\{
\begin{array}{ll}
(E_{\alpha_r} \oplus E_{\alpha_r})
\begin{bmatrix}
0 & -F_{\alpha_r}\\
 F_{\alpha_r} & 0 
\end{bmatrix}, & c + \alpha_r \textrm{ odd}\\
E_{\alpha_r} ((-1)^{\alpha} F_{\alpha_r}), & c + \alpha_r \textrm{ even}
\end{array}
\right..
\]
We can now write
\begin{align}
\label{U1TU1}
& (U^{-1})^{T} H U^{-1}  = ED, \qquad \qquad 
\Big(\mu_r=\left\{
\begin{array}{ll}
2m_r,  & c + \alpha_r \textrm{ odd} \\
m_r, & c + \alpha \textrm{ even}
\end{array}
\right.\Big),\\
E:=\bigoplus_{r=1}^{N}\big(\bigoplus_{j=1}^{\mu_r}  E_{\alpha_r}\big), & \qquad
D:=\bigoplus_{r=1}^{N}\big(\bigoplus_{j=1}^{m_r} 
D_r \big),
\quad
D_r:=\left\{
\begin{array}{ll}
\begin{bsmallmatrix}
0 & -F_{\alpha_r}\\
F_{\alpha_r} & 0 
\end{bsmallmatrix}, &  c + \alpha_r \textrm{ odd}\\
(-1)^{\alpha_r} F_{\alpha_r}, &  c + \alpha_r \textrm{ even}
\end{array}
\right..\nonumber
\end{align}

B applyng (\ref{U1TU1}), the equation (\ref{QTQcase2}) transforms to
%
\begin{align}\label{QTQ2c}
E D = &  E  X^{T}(E D) X\\
D = & ( E  X^{T} E ) D X.\nonumber
\end{align}
We conjugate (\ref{QTQ2c}) with $\Omega=\bigoplus_{r=1}^{N}\Omega_{\alpha_r,\mu_r}$ (see Lemma \ref{lemaP}) and manipulate it:
\begin{align}\label{ortoD2}
 \Omega^{T} D\Omega = & \big(\Omega^{T}E\Omega \big)\big(\Omega^{T} X\Omega\big)^{T}\big(\Omega^{T}E\Omega\big)\big(\Omega^{T}D \Omega\big)\big(\Omega^{T} X\Omega\big)\\
  \mathcal{D} = &\mathcal{F}\mathcal{X}^{T}\mathcal{F}\mathcal{D}\mathcal{X}; \nonumber
\end{align}
in which
\begin{align}
\label{eqD}
&\mathcal{D}:=  \Omega^{T} D \Omega=\bigoplus_{r=1}^{N} 
\big(
\bigoplus_{j=1}^{\alpha_r}
(-1)^{j-1}\Delta_r
\big),\quad
\Delta_r:=\left\{
\begin{array}{ll}
\hspace{-1mm}\bigoplus_{j=1}^{m_r}
\begin{bsmallmatrix}
0 & 1\\
-1 & 0 
\end{bsmallmatrix}, &  \hspace{-1mm} c + \alpha_r \textrm{ odd}\\
\hspace{-1mm}(-1)^{\alpha_r} I_{m_r}, & \hspace{-1mm} c + \alpha_r \textrm{ even}
\end{array}
\right.,\\
\label{eqXF}
& \mathcal{F}:= \Omega^{T}E\Omega=\bigoplus_{r=1}^{N} E_{\alpha_r}(I_{\mu_r}), \qquad \mathcal{X}:=\Omega^{T}X\Omega.
\end{align}
In view of a special form of $X$, obtained above when solving $\mathcal{J}(\mathcal{A})X=X\mathcal{J}(\mathcal{A})$, we conclude that $\mathcal{X}$ is of the form (\ref{0T0}) for $\alpha=(\alpha_1,\ldots,\alpha_N)$, $\mu=(\mu_1,\ldots,\mu_N)$.

Next, by conjugating with an appropriate permutation matrix, we transform blocks of $\mathcal{D}$ in (\ref{eqD}) either to the $\pm$ identities or the standard symplectic matrices:
\begin{align*}
\mathcal{B}:= & \Phi^{T}\mathcal{D}\Phi=\bigoplus_{r=1}^{N} 
\big(
\bigoplus_{j=1}^{\alpha_r}
(-1)^{j-1}B_r
\big),\qquad
B_r:=\left\{
\begin{array}{ll}
\begin{bsmallmatrix}
0 & I_{m_r}\\
-I_{m_r} & 0 
\end{bsmallmatrix}, & c + \alpha_r \textrm{ odd}\\
(-1)^{\alpha_r } I_{m_r}, & c + \alpha_r \textrm{ even}
\end{array}
\right.,\\
& \Phi=\bigoplus_{r=1}^{N} 
\big(
\bigoplus_{j=1}^{\alpha_r}
(-1)^{j-1}\Phi_r
\big),\qquad
\Phi_r:=\left\{
\begin{array}{ll}
\Omega_{2,m_r}, & c + \alpha_r \textrm{ odd}\\
I_{m_r}, &  c + \alpha_r \textrm{ even}
\end{array}
\right..
\end{align*}
Since $\Phi_r^{-1} E_{\alpha_r}(I_{\mu_r})\Phi_r=E_{\alpha_r}(I_{\mu_r})$,
the matrix $\mathcal{F}$ in (\ref{eqXF}) stays intact after conjugation by $\Phi$, so the equation (\ref{ortoD2}) becomes:
%
\begin{align}\label{ortoD2B}
  (\Phi \mathcal{B}\Phi^T) = &\mathcal{F}(\Phi \Phi^{-1})\mathcal{X}^{T}(\Phi^T(\Phi^T)^{-1})\mathcal{F}(\Phi\mathcal{B}\Phi^T)\mathcal{X} \nonumber\\
	 \mathcal{B} = & (\Phi ^{-1}\mathcal{F}\Phi)(\Phi^{-1}\mathcal{X}(\Phi^T)^{-1})^T((\Phi^T)^{-1}\mathcal{F}\Phi)\mathcal{B}(\Phi^T\mathcal{X}(\Phi^T)^{-1}) \\
	 \mathcal{B} = & \mathcal{F}\mathcal{Y}^T\mathcal{F}\mathcal{B}\mathcal{Y}, \qquad \quad \mathcal{Y}:=\Phi^T\mathcal{X}(\Phi^T)^{-1},\nonumber
\end{align}
%
This equation is of the form (\ref{eqFYFIY}), with $\mathcal{F}$ and $\mathcal{B}=\mathcal{C}$ as in (\ref{BBF}), and $\mathcal{X}=\mathcal{Y}$ of the form (\ref{0T0}), all for $\alpha=(\alpha_1,\ldots,\alpha_N)$, $\mu=(\mu_1,\ldots,\mu_N)$.

Finally, $H$-Orthogonal or $H$-symplectic solutions of $\mathcal{M}Q=Q\mathcal{M}$ are of the form:
\[
Q=U^{-1}\Omega (\Phi^T)^{-1}\mathcal{Y}\Phi^T\Omega^{T}U=\Psi^{-1}\mathcal{Y}\Psi, \qquad \Psi:=\Phi^T\Omega^{T}U,
\]
with $\mathcal{Y}$ as a solution of (\ref{ortoD2B}), which is provided by Lemmas \ref{EqT}, \ref{lemauni}.
\end{proof}

\begin{proof}[Proof of Theorem \ref{stabw1}]
Choose $c\in \{1,2\}$ and let $H$ and the $H$-skew-symmetric ($c=1$) or $H$-Hamiltonian ($c=2$) normal form $\mathcal{A}$, with one eigenvalue $\lambda\in \mathbb{C}\setminus\{0\}$, be:
\begin{align}
\label{H1l}
\mathcal{A}
=\bigoplus_{r=1}^{N}\Big( \bigoplus_{j=1}^{m_r}  \big( J_{\alpha_r}(\lambda)\oplus -(J_{\alpha_r}(\lambda))^T \big)\Big),\qquad 
H=\bigoplus_{r=1}^{N}\Big( \bigoplus_{j=1}^{m_r}  
\begin{bsmallmatrix}
0 & I_{\alpha_r}\\
(-1)^{c+1} I_{\alpha_r} & 0 
\end{bsmallmatrix} \Big).
\end{align}
Let $\Gamma_{\alpha_r}$ be as in (\ref{Jblock}), and denote by $\mathcal{J}(\mathcal{A})$ the Jordan canonical form of $\mathcal{A}^{}$ and by $U$ the corresponding transition matrix $U$ (i.e. $\mathcal{M}=U^{-1}\mathcal{J}(\mathcal{M}) U$):
\begin{align}\label{JCF1}
&\mathcal{J}(\mathcal{A})=\bigoplus_{r=1}^{N}\Big( \bigoplus_{j=1}^{m_r} \big( J_{\alpha_r}(\lambda)\oplus  J_{\alpha_r}(-\lambda)\big)\Big),\qquad 
U=\bigoplus_{r=1}^{N}\big( \bigoplus_{j=1}^{m_r} 
(I_{\alpha_r}\oplus  \Gamma_{\alpha_r})
\big).
\end{align}

We conjugate $\mathcal{J}(\mathcal{A})$ by an appropriate permutation matrix to collect together blocks corresponding to the same eigenvalue ($\lambda$ or $-\lambda$).
We use an analog of (\ref{perS}) (with $\alpha=2$) for block matrices. Define a $2m$-by-$2m$ block  permutation matrix whose blocks are identities and zero matrices:
\begin{equation}\label{perS2}
\Omega_{2,m}(I_{\alpha}):=\left[e_1(I_{\alpha})\;\,e_{3}(I_{\alpha})\;\,\ldots\;\,e_{2m-1}(I_{\alpha})\;\,e_2(I_{\alpha})\;\,e_{4}(I_{\alpha})\;\,\ldots\;\,e_{2m}(I_{\alpha})\right],
\end{equation}
in which $e_j(I_{\alpha})$ is a column ($n$ rows) with $I_{\alpha}$ in the $j$-th row and ${\alpha}$-by-${\alpha}$ zero matrices otherwise.
Post-multiplication by $\Omega_{2,m}(I_{\alpha})$ (with $\Omega_{2,2m}^{T}$) puts the 1st, the 3rd, \ldots, the $(2m-1)$-th column, the 2nd, the 4th,\ldots, the $2m$-th (row) together. Thus:
\begin{align}\label{J1F}
\widetilde{\Omega}_1^{T}\mathcal{J}(\mathcal{A})\widetilde{\Omega}_1=\bigoplus_{r=1}^{N}\Big( \bigoplus_{j=1}^{m_r} J_{\alpha_r}(\lambda)\oplus \bigoplus_{j=1}^{m_r} J_{\alpha_r}(-\lambda)\Big), \qquad \widetilde{\Omega}_1:=\bigoplus_{r=1}^{N} \Omega_{2,2m_r}(I_{\alpha_r}).
\end{align}
Similarly as in (\ref{perS2}), we can also permute blocks of possibly different dimensions:
\small
\begin{align}\label{perK2}
\widetilde{\Omega}_2=\left[e_1(I_{m_1\alpha_1})\;e_{3}(I_{m_2\alpha_2})\;\ldots\;e_{2N-1}(I_{m_N\alpha_N})\;e_2(I_{m_1\alpha_1})\;e_{4}(I_{m_2\alpha_2})\;\ldots\;e_{2N}(I_{m_N\alpha_N})
\right], 
\end{align}
\normalsize
where $e_j(I_{m_r\alpha_r})$ is a column ($2N$ rows) with $I_{m_r\alpha_r}$ in the $j$-th row and the appropriate zero matrices elsewhere.
Multiplication by $\widetilde{\Omega}_2$ (by $\widetilde{\Omega}_2^{T}$) from the right (left)
arranges columns (rows) in the same order as would $\Omega_{2,N}(I_{\alpha})$.
From (\ref{J1F}) we get
\begin{equation}\label{J2F}
\widetilde{\mathcal{J}}(\mathcal{A}):=\widetilde{\Omega}^{T}\mathcal{J}(\mathcal{M})\widetilde{\Omega}=
\bigoplus_{r=1}^{N}\big( \bigoplus_{j=1}^{m_r} J_{\alpha_r}(\lambda)\big)\oplus 
\bigoplus_{r=1}^{N}\big( \bigoplus_{j=1}^{m_r}  J_{\alpha_r}(-\lambda)\big), \quad 
\widetilde{\Omega}:=\widetilde{\Omega}_1\widetilde{\Omega}_2.
\end{equation}
%

By setting $\widetilde{U}=\widetilde{\Omega}^TU$, the equation $\mathcal{M}Q=Q\mathcal{M}$ (see (\ref{HQQH})) with $H$-orthogonal or $H$-symplectic $Q$, transforms to 
(\ref{JQQJ}) with $U$ replaced by $\widetilde{U}$:
\begin{equation}\label{wJQQwJ}
\widetilde{\mathcal{J}}(\mathcal{M})X=X\widetilde{\mathcal{J}}(\mathcal{M}), \qquad X=\widetilde{U} Q \widetilde{U}^{-1},\quad 
\qquad (\widetilde{U}=\widetilde{\Omega}^TU,\,\, Q^T H Q=H).
\end{equation}
Proposition \ref{resAoXXA} gives the solution $X=X_1\oplus X_2$ of (\ref{wJQQwJ}), where $X_1,X_2$ are $N$-by-$N$ block matrices such that their blocks $(X_1)_{rs}$, $(X_2)_{rs}$ are $m_r$-by-$m_s$ block matrices whose blocks are rectangular Toeplitz of size $\alpha_r\times \alpha_s$
and of the form (\ref{QTY}).

In view of (\ref{QK}) for $U=\widetilde{U}$, now $H$-orthogonality or $H$-symplecticity of $Q$ yields:
%
\begin{align}\label{QTQHH}
\widetilde{\Omega}^{T}(U^{-1})^{T} H U^{-1}\widetilde{\Omega} =   &  X^T \big(\widetilde{\Omega}^{T}(U^{-1})^{T} H U^{-1} \widetilde{\Omega}\big) X, \qquad \quad
(Q=\widetilde{U}^{-1} X \widetilde{U}).
\end{align}
We again use $\Gamma_{\alpha}^{-1}=(-1)^{\alpha +1}\Gamma_{\alpha}^{}$, but in contrast to (\ref{UTHU}), we now derive 
\begin{align}\label{s2aH}
( (I_{\alpha} \oplus \Gamma_{\alpha})^{-1})^T\begin{bmatrix}
0 & I_{\alpha} \\
(-1)^{c+1} I_{\alpha} & 0
\end{bmatrix}(I_{\alpha} \oplus \Gamma_{\alpha})^{-1}
= &
\begin{bmatrix}
0 & (-1)^{\alpha-1}\Gamma_{\alpha}\\
(-1)^{c+\alpha} \Gamma_{\alpha}^{T} & 0 
\end{bmatrix}
=
\\
=
\big((-1)^{\alpha-1}\Gamma_{\alpha}\oplus (-1)^{c+\alpha}\Gamma_{\alpha}^T\big)
\begin{bmatrix}
0 & I_{\alpha}\\
I_{\alpha} & 0 
\end{bmatrix}.  \nonumber
\end{align}
Using (\ref{s2aH}) we now decompose $(U^{-1})^{T} H U^{-1} $ (the trick of the proof):
\begin{align}\label{deDVH}
&(U^{-1})^{T} H U^{-1}
=\bigoplus_{r=1}^{N}\left(\bigoplus_{j=1}^{m_r} 
\Big(
\big((-1)^{\alpha_r-1}\Gamma_{\alpha_r}\oplus (-1)^{c+\alpha_r}\Gamma_{\alpha_r}^T\big)
\begin{bmatrix}
0 & I_{\alpha_r}\\
I_{\alpha_r} & 0 
\end{bmatrix}
\Big)\right)
=D K,\\
& D:=\bigoplus_{r=1}^{N}\Big(\bigoplus_{j=1}^{m_r}  \big((-1)^{\alpha_r-1}\Gamma_{\alpha_r}\oplus (-1)^{c+\alpha_r}\Gamma_{\alpha_r}^T \big)  \Big), \qquad 
K: =\bigoplus_{r=1}^{N}\big(\bigoplus_{j=1}^{m_r}  
\begin{bsmallmatrix}
0 & I_{\alpha_r}\\
I_{\alpha_r} & 0 
\end{bsmallmatrix}
 \big).
\end{align}
Furthermore, (\ref{deDVH}) yields the decomposition of $\widetilde{\Omega}^{T} \big((U^{-1})^{T} H U^{-1}\big) \widetilde{\Omega}$:
\begin{align}
\label{BK}
&\widetilde{\Omega}^{T} \big((U^{-1})^{T} H U^{-1}\big) \widetilde{\Omega}
=(\widetilde{\Omega}^{T} D \widetilde{\Omega})(\widetilde{\Omega}^{T} K \widetilde{\Omega})=B \widetilde{K},
\qquad n:=\sum_{r=1}^{N}\alpha_r m_r,\\
\widetilde{K}:=\widetilde{\Omega}^{T} V \widetilde{\Omega} & =
\begin{bmatrix}
0 & I_{n}\\
I_{n} & 0 
\end{bmatrix},\,\,\,
B:=\widetilde{\Omega}^{T} D \widetilde{\Omega}
=B_1 \oplus (-1)^{c+1} B_1^{T}, \,\,\, B_1:=\bigoplus_{r=1}^{N}\big(\bigoplus_{j=1}^{m_r} (-1)^{\alpha_r-1} \Gamma_{\alpha_r}  \big).\nonumber
\end{align}
By applying (\ref{BK}) to (\ref{QTQHH}) we obtain:
\begin{align}\label{QTQH2}
B \widetilde{K} = &  X^{T}  B \widetilde{K} X \nonumber\\
 B_1 \oplus (-1)^{c+1} B_1^{T} = & (X_1^{T} \oplus X_2^{T}) ( B_1 \oplus (-1)^{c+1} B_1^{T}) (\widetilde{K} (X_1 \oplus X_2) \widetilde{K}^{-1})\\
B_1 \oplus  (-1)^{c+1} B_1^{T} = & (X_1^{T} \oplus X_2^{T}) (  B_1 \oplus (-1)^{c+1} B_1^{T}) (X_2 \oplus X_1)\nonumber\\
B_1 = &  X^{T}_1 B_1 X_2,\qquad \quad ((-1)^{c+1} B_1^{T}=X^{T}_2 \big(\,(-1)^{c+1} B_1^{T}) X_1 \,\big). \nonumber
\end{align}

By conjugating both hand sides of the last equation of (\ref{QTQH2}) by $\Omega$ from Lemma \ref{lemaP}, and then slightly manipulating it,
we deduce
\begin{align}\label{ortoD3}
  \Omega^{T}B_1\Omega = & (\Omega^{T} X_1\Omega)^{T}(\Omega^{T}B_1\Omega)(\Omega^{T} X_2\Omega)\\
\mathcal{B}_1= &  \mathcal{X}^{T}_1\mathcal{B}_1 \mathcal{X}_2 \qquad \qquad (\mathcal{X}_1:=\Omega^{T} X_1\Omega, \quad \mathcal{X}_2:=\Omega^{T} X_2\Omega),\nonumber
\end{align}
where $\mathcal{B}_1:= \bigoplus_{r=1}^{N}\big( 
(-1)^{\alpha_r-1} \Gamma_{\alpha}(I_{m_r})
\big)$, in which 
$
\Gamma_{\alpha_r}(I_{m_r})
$ is a block exchange matrix with altermating $\pm I_{m_r}$ on the anti-diagonal (see (\ref{Jblock})).
We observe that the form of the solution $X=X_1\oplus X_2$ of (\ref{wJQQwJ}) yields $\mathcal{X}_1, \mathcal{X}_2\in \mathbb{T}^{\alpha,\mu}$ (cf. (\ref{0T0}), Proposition \ref{lemanilpo}). 
Therefore, $H$-orthogonal or $H$-symplectic solution $Q$
of $\mathcal{A}Q=Q\mathcal{A}$ is of the form
\begin{align}\label{QPsi}
&Q=U^{-1}\widetilde{\Omega}(X_1\oplus X_2)\widetilde{\Omega}^{T} U
=\Psi^{-1} \big(\mathcal{X}_1\oplus (\mathcal{X}^{T}_1)^{-1}\big)\Psi, \qquad \mathcal{X}_1\in \mathbb{T}^{\alpha,\mu},\\
&\qquad \Psi:=(I\oplus \mathcal{B}_1)\big((\Omega^{T}\oplus \Omega^{T})\widetilde{\Omega}^{T} U\big).\nonumber
\end{align}
This completes the proof of the theorem.
\end{proof}

\begin{remark}\label{Rfw}
In the real case the situation is somewhat more involved. Real normal forms are more numerous and, to some extent, structurally more complicated; see \cite{Laub}, \cite{LanRod}. The corresponding centralizers and isotropy groups are considered with respect to real $H$-orthogonal or $H$-simplectic similarity. Applying the same general approach as in this paper, however, require solving matrix equations over reals, which is a significant difference. This issue will be addressed in future work.
\end{remark}

\section{Acknowledgement}

The research was supported by 
Slovenian Research and Innovation Agency (grants no. P1-0291 and no. J1-3005).




\begin{thebibliography}{80}

\small
\baselineskip=10pt












\bibitem{TeranDopi2}
F. De Teran, F. M. Dopico,
The solution of the equation $XA+ AX^T= 0$ and its application to the theory of orbits,
Lin. Alg. Appl. 434 (1), 44-67.


\bibitem{DK}
A. Dmytryshyn, B. K\aa gstr\"{o}m,
Coupled Sylvester-type Matrix Equations and Block Diagonalization,
SIAM J. Matrix Anal. Appl. 36 (No. 2) (2015),  580-593.


\bibitem{DKS}
A. Dmytryshyn, B. K\aa gstr\"{o}m, V. V. Sergeichuk,
Skew-symmetric matrix pencils: Codimension counts and the solution of a pair of matrix equations,
Linear Algebra Appl. 438 (No. 8) (2013), 3375-3396.




\bibitem{Djok2} D.\v{Z}. \DJ okovi\'v, K. Rietsch, K. Zhao, 
Normal forms for orthogonal similarity classes of skew-symmetric matrices,
J. Algebra 308 (2007) 686-703.






\bibitem{Gant}
F. R. Gantmacher,
The theory of matrices,
Chelsea Publishing Company, New York, 1959. 








\bibitem{HornJohn} 
R. A. Horn, C. R. Johnson,
Matrix analysis,
Cambridge University Press, Cambridge, 1990.



\bibitem{HornMerino}
R. A. Horn, D. I. Merino,
Contragradient equivalence: a canonical form and some applications,
Linear Algebra Appl., 214 (1995), pp. 43-92.




\bibitem{Hua45}
L. K. Hua,
Geometries of matrices I. Generalizations of von Staudt's theorem,
Trans. Amer. Math. Soc. 57 (1945), 441-481.






\bibitem{LanRod}
P. Lancaster, L. Rodman,
Algebraic Riccati Equations,
Clarendon Press, Oxford, 1995.


\bibitem{Laub}
A. J. Laub, K. Meyer,
Canonical forms for symplectic and Hamiltonian matrices. 
Celestial Mechanics 9 (1974), 213-238.




\bibitem{Lin}
W. W. Lin, V. Mehrmann, H. Xu,
Canonical Forms for Hamiltonian and Symplectic Matrices and Pencils,
Linear Algebra Appl. 302-303 (1999), 469-533.      


\bibitem{Seitz}
M. W., Liebeck, G. M. Seitz, 
Unipotent and nilpotent classes in simple algebraic groups and lie algebras, 
Providence, American Mathematical Society, 2012.






\bibitem{Mehl}
C. Mehl,
On classification of normal matrices in indefinite inner product spaces,
Electron. J. Linear Algebra 15 (2006), 50-83.

\bibitem{Milne}
J. S. Milne,
Algebraic Groups: The Theory of Group Schemes of Finite Type over a Field,
Cambridge: Cambridge University Press, 2017.















\bibitem{Spring}
T.A. Springer, R. Steinberg, Conjugacy classes. In: Seminar on Algebraic Groups and Related Finite Groups. Lecture Notes in Mathematics, vol 131. Springer, Berlin, Heidelberg, 1970.


\bibitem{TSH}
T. Star\v{c}i\v{c},
Hong's canonical form of a Hermitian matrix
with respect to orthogonal *congruence,
Linear Algebra Appl. 630 (2021), 241-251.


\bibitem{TSOC}
T. Star\v{c}i\v{c},
Isotropy groups of the action of orthogonal *congruence on Hermitian matrices, 
Linear Algebra Appl. 694 (2024), 101-135.


\bibitem{TSOS}
T. Star\v{c}i\v{c},
Isotropy groups of the action of orthogonal similarity on symmetric matrices,
Linear Multilinear Algebra. 71 (no. 5) (2023), 842-866.






\bibitem{Wan}
Z. Wan,
Geometry of matrices,
World Scientific, New York Heidelberg Berlin, 1996.


\bibitem{Well}
J. Wellstein, \" Uber symmetrische, alternierende und orthogonale Normalformen von Matrizen. J. f\" ur Reine Angew. Math., vol. 1930, no. 163, 1930, pp. 166-182.


\bibitem{Weyl}
H. Weyl,
The classical groups, their invariants and representations,
Princeton University Press, 1946. 

\end{thebibliography}
\end{document}